\documentclass [11pt] {article}
\usepackage{amsfonts}
\usepackage{amssymb}
\usepackage{amsmath}
\usepackage{times}
\usepackage{color}
\usepackage{graphicx}
\usepackage{enumerate}

\newtheorem{lemma}{Lemma}[section]

\newtheorem{theorem}{Theorem}

\newtheorem{question}{Question}
\newtheorem{remark}{Remark}[section]

\newtheorem{proposition}{Proposition}[section]
\newtheorem{definition}{Definition}[section]

\def\<{\langle}
\def\>{\rangle}

\def\to{\rightarrow}

\begin{document}

\title{On large girth regular graphs and random processes on trees}
\author{{\'Agnes Backhausz   ~~ Bal\'azs Szegedy} \\  \textsc{\footnotesize MTA Alfr\'ed R\'enyi Institute of Mathematics.} {\footnotesize Re\'altanoda utca 13--15., Budapest, Hungary, H-1053}\\ \texttt{\footnotesize backhausz.agnes@renyi.mta.hu, szegedy.balazs@renyi.mta.hu}}

\maketitle

\abstract{We study various classes of random processes defined on the regular tree $T_d$ that are invariant under the automorphism group of $T_d$. Most important ones are factor of i.i.d. processes (randomized local algorithms), branching Markov chains and a new class that we call typical processes. Using Glauber dynamics on processes we give a sufficient condition for a branching Markov chain to be factor of i.i.d. Typical processes are defined in a way that they create a correspondence principle between random $d$-reguar graphs and ergodic theory on $T_d$. Using this correspondence principle together with entropy inequalities for typical processes, we prove that there are no approximative covering maps from random $d$-regular graphs to $d$-regular weighted graphs.}
\bigskip

\begin{small}


\end{small}

\section{Introduction} 

F\"urstenberg's correspondence principle creates a fruitful link between finite combinatorics and ergodic theory. It connects additive combinatorics with the study of shift invariant measures on the Cantor set $\{0,1\}^\mathbb{Z}$. In particular it leads to various strengthenings and generalizations of Szemer\'edi's celebrated theorem on arithmetic progressions. 

The goal of this paper is to study a similar correspondence principle between finite large girth $d$-regular graphs and ${\rm Aut}(T_d)$ invariant probability measures on $F^{V(T_d)}$ where $F$ is a finite set and $T_d$ is the $d$-regular tree with vertex set $V(T_d)$. The case $d=2$ is basically classical ergodic theory however the case $d\geq 3$ is much less developed. 

Our approach can be summarized as follows. Assume that $G$ is a $d$-regular graph of girth $g$. We think of $d$ as a fixed number (say $10$) and $g$ as something very large. We wish to scan the large scale structure of $G$ in the following way. We put a coloring $f:V(G)\rightarrow F$ on the vertices of $G$ with values in a finite set $F$. (It does not have to be a proper coloring i.e. neighboring vertices can have identical color.) Then we look at the colored neighborhoods (of bounded radius) of randomly chosen points $v\in V(G)$. By this sampling we obtain a probability distribution on $F$-colored (bounded) trees that carries valuable information on the global structure of $G$. For example, if there is a coloring $f:V(G)\rightarrow\{0,1\}$ such that, with high probability, a random vertex $v$ has a color different from its neighbours, then $G$ is essentially bipartite.  

It turns out to be very convenient to regard the information obtained from a specific coloring as an approximation of a probability measure on $F^{V(T_d)}$ that is invariant under ${\rm Aut}(T_d)$. This can be made precise by using Benjamini--Schramm limits of colored graphs (see Section \ref{invproc}, or \cite{bs} for the original formulation). We will use the following definition.

\begin{definition} Let $\mathcal{S}=\{G_i\}_{i=1}^\infty$ be a sequence of $d$-regular graphs. We say that $\mathcal{S}$ is a large girth sequence if for every $\varepsilon>0$ there is an index $n$ such that for every $i\geq n$ the probability that a random vertex in $G_i$ is contained in a cycle of length at most $\lceil 1/\varepsilon\rceil$ is at most $\varepsilon$.
\end{definition}

\begin{definition}\label{profile} Let $\mathcal{S}=\{G_i\}_{i=1}^\infty$ be a large girth sequence of $d$-regular graphs, and $F$ a finite set. We denote by $[\mathcal{S}]_F$ the set of ${\rm Aut}(T_d)$ invariant probability measures on $F^{V(T_d)}$ that arise as Benjamini--Schramm limits of $F$-colorings $\{f_i:V(G_i)\rightarrow F\}_{i=1}^\infty$ of $\mathcal{S}$. We denote by $[\mathcal{S}]$ the set $\bigcup_{n\in\mathbb{N}}[\mathcal{S}]_{\{1,2,\dots,n\}}$.
\end{definition}

 It is clear that if $\mathcal S'$ is a subsequence of $\mathcal S$, then $[\mathcal S]\subseteq [\mathcal S']$. If $[\mathcal{S}]=[\mathcal S']$ holds for every subsequence $S'$ of $S$, then $\mathcal{S}$ is called {\it local-global convergent} (see Subsection \ref{corresp} and \cite{HLSz}). Local-global convergent sequences of graphs have limit objects in the form of a {\it graphing} \cite{HLSz}. For a convergent sequence $\mathcal S$ the set $[\mathcal{S}]$ carries important information on the structure of the graphs in $\mathcal{S}$.

We call a process $\mu$  {\it universal} if $\mu\in[\mathcal{S}]$ for every large girth sequence $\mathcal{S}$.  Universality means, roughly speaking, that it defines a structure that is universally present in every large girth $d$-regular graph.
  Weakening the notion of universality, we call a process $\mu$ {\it typical} if $\mu\in [\{\mathbb G_{n_i}\}_{i=1}^\infty]$ holds with probability 1 for some fixed sequence $\{n_i\}_{i=1}^{\infty}$,  where $\{\mathbb G_{n_i}\}_{i=1}^{\infty}$ is a sequence of independently and uniformly chosen random $d$-regular graphs  with $|V(\mathbb G_{n_i})|=n_i$.   We will see that understanding typical processes is basically equivalent with understanding the large scale structure of random $d$-regular graphs. More precisely, we will formulate a correspondence principle (see Subsection \ref{corresp}) between the properties of random $d$-regular graphs and typical processes.  

\medskip

Among universal processes, factor of i.i.d processes on $T_d$ (see \cite{russ} and the references therein) have a distinguished role because of their close connection to local algorithms \cite{gamarnik, HLSz, kungabor}. They can be used to give estimates for various structures (such as large independent sets \cite{csoka, harangi, hoppen, mustazee}, matchings \cite{csokalipp, nazarov}, subgraphs of large girth \cite{damien, kungabor}, etc., see also \cite{goldberg}) in $d$-regular graphs. On the other hand, \cite{cordec} characterizes the  covariance structure of weak limits of factor of i.i.d. processes and thus it gives a necessary condition for a process to be factor of i.i.d. However, there are only few general and widely applicable   sufficient conditions. This is a difficult question even for branching Markov processes that are important in statistical physics (e.g. Ising model, Potts model). In Section \ref{glauber} we give a Dobsrushin-type sufficient condition for a branching Markov chain to be factor of i.i.d. 
We use standard methods from statistical physics, in particular, a heat-bath version of Glauber dynamics. The idea behind this goes back to Ornstein and Weiss: sufficient conditions for fast mixing of Glauber dynamics  often imply that the process is factor of i.i.d. See also the paper of H\"aggstr\"om, Jonasson and Lyons \cite{russregi}.
We will see that the necessary condition on the covariance structure given in \cite{cordec} is not sufficient for a branching Markov chain to be factor of i.i.d.  To show this, we use  our necessary conditions for typical processes (Section \ref{entropy}), which automatically apply for factor of i.i.d. processes.

\medskip

Our paper is built up as follows. In the first part we summarize various known and new facts about factor of i.i.d, universal and typical processes, local-global convergence and graphings. Moreover, in this part, we formulate our correspondence principle between typical processes and random $d$-regular graphs. In Section \ref{glauber} we focus more on branching Markov chains on $T_d$. We give a Dobrushin-type sufficient condition for a branching Markov chain to be factor of i.i.d.  In the last part (Section \ref{entropy}) we give necessary conditions for a process to be typical using joint entropy functions. We will see that this result implies necessary conditions on the large scale structure of random $d$-regular graphs. (Note that our entropy method is closely related to the F-invariant, introduced by Lewis Bowen \cite{lewis} in ergodic theory, and also to the ideas developed by Molloy and Reed \cite{molloyreed} to study random $d$-regular graphs in combinatorics.) In particular, we prove  that the value distributions of eigenvectors of random $d$-regular graphs can not be concentrated around boundedly many values (this is even true for approximative eigenvectors). Moreover, we show that random $d$-regular graphs do not cover bounded $d$-regular weighted graphs (for precise formulation, see Theorem \ref{thm:combap}).  These  results are closely related to the papers of Molloy and Reed \cite{molloyreed} about dominating ratio and Bollob\'as \cite{bollind} about independence numbers.

\section{Invariant processes}\label{invproc}

Let $T_d$ be the (infinite) $d$-regular tree with vertex set $V(T_d)$ and edge set $E(T_d)$. 
Let $M$ be a topological space. We denote by $I_d(M)$ the set of $M$-valued random processes on the $d$-regular tree 
$T_d$ that are invariant under automorphisms of $T_d$. More precisely, $I_d(M)$ is the set of ${\rm Aut}(T_d)$ 
invariant Borel probability measures on the space $M^{V(T_d)}$. (If $\Psi\in {\rm Aut}(T_d)$, then $\Psi$ induces 
a map naturally from $M^{V(T_d)}$ to itself: given a labelling of the vertices of $T_d$, the new label of a vertex is 
the label of its inverse image at $\Psi$. The probability measures should be invariant with respect to this induced map.) 
The set $I_d(M)$ possesses a topological structure; namely the restriction of the weak topology for probability measures on $M^{V(T_d)}$ to $I_d(M)$. Note that most of the time in this paper $M$ is a finite set. We denote by $I_d$ the set of invariant processes on $T_d$ with finitely many values.

Let $T_d^*$ denote the rooted $d$-regular tree: it is $T_d$ with a distinguished vertex $o$, which is called the root. 
Let $N$ be a topological space and $f:M^{V(T_d^*)}\rightarrow N$ be a Borel measurable function that is invariant under 
${\rm Aut}(T_d^*)$, which is the set of root-preserving automorphisms of $T_d^*$. For every $\mu\in I_d(M)$ the function $f$ defines a new process $\nu \in I_d(N)$ by evaluating 
$f$ simultaneously at every vertex $v$ (by placing the root on $v$) on a $\mu$-random element in $M^{V(T_d)}$.
We say that $\nu$ is a {\it factor} of $\mu$. 

A possible way to get processes in $I_d$ goes through Benjamini--Schramm limits. For the general definition see \cite{bs}. We will use and formulate it for colored large-girth graph sequences, as follows. Let $F$ be a finite set. Assume that $\{G_i\}_{i=1}^\infty$ is a large girth sequence of $d$-regular graphs. Let $\{f_i:V(G_i)\rightarrow F\}_{i=1}^\infty$ be a sequence of colorings of $G_i$. For every pair of numbers $r,i\in\mathbb{N}$ we define the probability distribution $\mu_{r,i}$ concentrated on rooted $F$-colored finite  graphs as follows. We pick a random vertex $v\in V(G_i)$ and then we look at the neighborhood $N_r(v)$ of radius $r$ of $v$ (rooted by $v$) together with the coloring $f_i$ restricted to $N_r(v)$. The colored graphs $(G_i,f_i)$ are Benjamini--Schramm convergent if  for every $r\in\mathbb{N}$ the sequence $\{\mu_{r,i}\}_{i=1}^\infty$ weakly converges to some measure $\mu_r$. The limit object is the probability measure $\mu$ on  $F^{V(T_d^*)}$ with the property that the marginal of $\mu$ in the neighborhood of radius $r$ of the root is $\mu_r$. It is easy to see that the measure we get from $\mu$ by forgetting the root is in $I_d(F)$. 

We list various classes of invariant processes on $T_d$ that are related to large girth sequences of finite graphs. 

\bigskip

\noindent{\bf Factor of i.i.d. processes:}~Let $\mu\in I_d([0,1])$ be the uniform distribution on $[0,1]^{V(T_d)}$, which is 
the product measure of the uniform distributions on the interval $[0,1]$. A {\it factor of i.i.d. process is a factor of the process $\mu$}. Let $F_d$ denote the set of such processes in $I_d$. 
See Lemma \ref{ebred} for an easy example producing independent sets as factor of i.i.d. processes.

\bigskip

\noindent{\bf Local processes:}~We say that a process is {\it local} if it is in the closure of factor of i.i.d 
processes in the weak topology. Let $L_d$ denote the set of such processes in $I_d$.

\bigskip

\noindent{\bf Universal processes:}~A process $\mu\in I_d$ is called universal if $\mu\in [\mathcal{S}]$ holds for every large girth sequence $\mathcal{S}$ of $d$-regular graphs.  We denote the set of such processes by $U_d$. 

\bigskip

\noindent{\bf Typical processes:}~A process  $\mu\in I_d$ is called typical if $\mu\in [\{\mathbb G_{n_i}\}_{i=1}^\infty]$ holds with probability 1 for some fixed sequence $\{n_i\}_{i=1}^{\infty}$,  where $\{\mathbb G_{n_i}\}_{i=1}^{\infty}$ is a sequence of independently chosen uniform random $d$-regular graphs  with $|V(\mathbb G_{n_i})|=n_i$. We denote the set of typical processes by $R_d$. 

\bigskip

\begin{lemma} \label{lem:bovul}We have the follwing containments:

$$F_d\subseteq L_d\subseteq U_d\subseteq R_d.$$

\end{lemma}

\begin{proof} The first and last containments are trivial. The containment $L_d\subseteq U_d$ is easy to see. For a proof we refer to \cite{HLSz} where a much stronger theorem is proved. \hfill $\square$
\end{proof}

\medskip

We also know by recent results of Gamarnik and Sudan \cite{gamarnik} and Rahman and Vir\'ag \cite{mustazee} that $L_d\neq R_d$ for 
sufficiently large $d$. Their result implies that the indicator function of a maximal independent set 
(a set of vertices that does not contain any neighbors)  in a random $d$-regular graph is not in $L_d$ (that is, the largest independent set can not be approximated with 
factor of i.i.d. processes); on the other hand, it is in $R_d$.

It is sometimes useful to consider variants of $F_d,L_d,U_d$ and $R_d$ where the values are in an infinite topological space $N$. The definitions can be easily modified using the extension of Benjamini--Schramm limits to colored graphs where the colors are in a topological space. We denote by $F_d(N),L_d(N),U_d(N)$ and $R_d(N)$ the corresponding set of processes.
Using this notation, it was proved in \cite{harangi} that $F_d(\mathbb{R})\neq L_d(\mathbb{R})$. In that paper Harangi and Vir\'ag used random Gaussian wave functions \cite{wave} to show this.  
  See also Corollary 3.3. in the paper of Lyons 
\cite{russ}: it provides a discrete-valued example for a process in $L_d(\lbrace 0,1\rbrace)\setminus U_d(\lbrace 0,1\rbrace)$.
\medskip

The following question remains after these results.

\begin{question} Is it true that $U_d=L_d$?~ Is it true that $U_d=R_d$? 
\end{question}

\medskip

It is an important goal of this paper to give sufficient conditions (for particular models) and necessary conditions for processes to be in one of  the above classes.
A recent result \cite{cordec} in this direction is the following.

\begin{theorem}\label{thmcordec} Let $\mu\in L_d(\mathbb{R})$ and let $v,w\in V(T_d)$ be two vertices of distance $k$. Let $f:T_d\rightarrow\mathbb{R}$ be a $\mu$-random function. Then the correlation of $f(v)$ and $f(w)$ is at most $(k+1-2k/d)(d-1)^{-k/2}$. 
\end{theorem}

Note that the statement also holds for processes in $R_d$; however the proof of that extension uses the very hard theorem of J. Friedman \cite{friedman} on the second eigenvalue of random $d$-regular graphs. There are various examples showing that the condition of Theorem \ref{thmcordec} is not sufficient. We also give a family of such examples using branching Markov processes (see Theorem \ref{exnotsuff}).
Branching Markov processes will play an important role in this paper so we give a brief description of them.

\medskip

\noindent{\bf Branching Markov processes:}~Now choose $M$ to be a finite state space $S$ with the 
discrete topology. Let $Q$ be the transition matrix of a reversible Markov chain on the 
state space $S$. Choose the state of the root uniformly at random. Then make random steps according 
to the transition matrix $Q$ to obtain the states of the neighbors of the root. These steps 
are made conditionally independently, given the state of the root. Continue this: given the 
state of a vertex at distance $k$ from the root, choose the states of its neighbors which are 
at distance $k+1$ from the root conditionally independently and according to the transition matrix $Q$.
It is easy to see that reversibility implies that the distribution of the collection of the random variables we get is invariant, 
hence the distribution of the branching Markov process (which will be denoted by $\nu_Q$) is in $I_d(S)$.

In the particular case when there is a fixed probability 
of staying at a given state, and another fixed probability of transition between distinct states, the branching Markov process is identical to  
the Potts model on the tree and for $|S|=2$ we get the Ising model. See e.g. \cite{evans, sly} for the description of the connection of the parameters 
of the two models. 

\medskip

\subsection{Correspondence between typical processes and random $d$-regular graphs}

\label{corresp}

Typical processes might be of interest on their own, being the processes that can be modelled on random $d$-regular graphs. In addition to this, we can go in the other direction. As we will see later, results on typical processes imply statements for random $d$-regular graphs. In the last section, based on entropy estimates we give necessary conditions for an invariant process to be typical. In this section we show how these results can be translated to statements about random $d$-regular graphs. We will present a correspondence principle between these objects. 

\subsubsection{Local-global convergence and metric}

When we want to study the correspondence between typical processes (which are defined on 
the vertex set of the $d$-regular tree) and random $d$-regular graphs, another notion of convergence of bounded 
degree graphs will be useful. In this subsection we briefly resume the concept of local-global convergence (also called colored neighborhood convergence) based on the papers of Bollob\'as and Riordan \cite{BR} (where this notion was introduced) and Hatami, Lov\'asz and Szegedy \cite{HLSz}. 

In the beginning of this section, we defined the notion of local (Benjamini--Schramm) convergence of  bounded degree graphs. However, we need a finer convengence notion that captures more of the global structure than local convergence.
 Recall that if $F$ is a finite set (colors) and $G$ is a finite graph with some $f: V(G)\rightarrow F$ , then by picking a random vertex $v\in V(G)$ and looking at its neighborhood $N_r(v)$ of radius $r$, we get a probability distribution $\mu_{r, G,f}$, which is concentrated on rooted $F$-colored finite graphs. (These distributions are called the local statistics of the coloring $f$.)  
Let $[k]=\lbrace 1, \ldots, k\rbrace$, and we define 
\[Q_{r,G,k}=\lbrace \mu_{r,G,f}\vert f:V(G)\rightarrow [k]\rbrace. \] 

Let $U^{r,k}$ be the set of triples $(H, o, f)$ where $(H, o)$ is a rooted graph of radius at most $r$ and $f: V(H)\rightarrow [k]$ is a coloring of its vertices with (at most) $k$ colors. Let $\mathcal M(U^{r,k})$ be the set of probability measures on $U^{r,k}$. With this notation, we have that $Q_{r, G, k}\subseteq \mathcal M(U^{r,k})$. The space $\mathcal M(U^{r,k})$ is a compact metric space equipped with the total variation 
distance of probability measures: 
\[d_{TV}(\mu, \nu)=\sup_{A\subseteq U^{r,k}}|\mu(A)-\nu(A)|.\]

(Note that we will use an equivalent definition of total variation distance later in this paper.)

\begin{definition}[Local-global convergence, \cite{HLSz}.] A sequence of finite graphs $(G_n)_{n=1}^{\infty}$ with uniform degree bound $d$ is locally-globally convergent if for every $r, k\geq 1$, the sequence $(Q_{r,G_n, k})$ converges in the Hausdorff distance inside the compact metric space $(\mathcal M(U^{r,k}),\, d_{TV})$.
\end{definition}

For every locally-globally convergent sequence $(G_n)$ of bounded degree graphs there is a limit object called graphing such that the sets of local statistics of $G_n$ converge to the local stastics of the limit object; see Theorem 3.2 of \cite{HLSz} for the precise statement, and e.g. \cite{aldous, cordec, gabor} for more about graphings. 

 The following metrization of local-global convergence was defined by Bollob\'as and Riordan \cite{BR}.

\begin{definition}[Colored neighborhood metric, \cite{BR}]
 Let $G, G'$ be finite graphs. Their colored neighborhood distance is the following:
\begin{equation}\label{dcn}d_{CN}(G,G')=\sum_{k=1}^{\infty}\sum_{r=1}^{\infty} 2^{-k-r} d_H(Q_{r,G,k}, Q_{r, G',k}),\end{equation}
where $d_H$ denotes the Hausdorff distance of sets in the compact metric space $(\mathcal M(U^{r,k}),\, d_{TV})$. 
\end{definition}

Let $X_d$ be the set of all finite graphs with maximum degree at most $d$. It  is clear from the definition that every 
sequence in $X_d$ contains a locally-globally convergent subsequence \cite{HLSz}. It follows that the completion 
$\overline {X_d}$ of the metric space $(X_d, d_{CN})$ is a compact metric space. It was proved in \cite{HLSz} that the elements of $X_d$ can be represented by certain measurable graphs called 
graphings. 

\begin{definition} [Graphing, \cite{HLSz}.]Let $\Omega$ be a Polish topological space and let $\nu$ be a probability measure on the Borel sets in $X$. A graphing is a graph $\mathcal G$ on $V(\mathcal G)=\Omega$ with Borel measureable edge set $E(\mathcal G)\subset \Omega\times \Omega$ in which all degrees are at most $d$ and 
\[\int_A e(x, B)d\nu(x)=\int_B e(x, A)d\nu(x)\] 
for all measurable sets $A, B\subset \Omega$, where $e(x, S)$ is the number of edges from 
$x\in\Omega$ to $S\subseteq \Omega$.
\end{definition}
If $\mathcal G$ is graphing, then $Q_{r, \mathcal G, k}$ makes sense with the additional condition that the coloring $f: \Omega\rightarrow [k]$ is measurable. Hence local-global convergence and metric both extend to graphings. 

We will need the following two lemmas about the metric $d_{CN}$. We remark that for sake of simplicity we will use the notion of random $d$-regular graphs with $n$ vertices in the sequel without any restriction on $d$ and $n$. If $d$ and $n$ are both odd, then there are no such graphs. We will formulate the statements such that they trivially hold for the empty set as well.

\begin{lemma} \label{lem:halo}For all $d\geq 1$ and $\varepsilon>0$ there exists $F(\varepsilon)$ such that for all $n\geq 1$ in 
the set of $d$-regular graphs with $n$ vertices endowed with $d_{CN}$ there exist an $\varepsilon$-net of size at most $F(\varepsilon)$. \label{lem:net}
\end{lemma} 
\begin{proof}Using compactness, we can choose an $\varepsilon/2$-net $N$ in the space $(\overline {X_d}, d_{CN})$. We show that  $F(\varepsilon):=|N|$ is a good choice. Let $N'$ be the subset of $N$ consisting of points $x$ such that the ball of radius 
$\varepsilon/2$ around $x$ contains a $d$-regular graph with $n$ vertices. To each element in $N'$ we assign a $d$-regular graph with $n$ vertices of distance at most $\varepsilon/2$. It is clear that set of these 
graphs  have the desired properties.   \end{proof}\hfill $\square$

\begin{lemma}\label{lem:lip}
For all $\delta>0$ there exists $i_0$ such that for all $i\geq i_0$ and graphs $G_1, G_2\in X_d$ both on the vertex set  $[i]$ and $|E(G_1)\triangle E(G_2)|=1$ satisfy  
$d_{CN}(G_1, G_2)\leq \delta$.
\end{lemma}
\begin{proof} 
Since the sum of the weights is finite in \eqref{dcn}, and the all the Hausdorff distances are at most 1, it is enough to prove the statement for a single term. Let us fix $k$ and $r$. Let  $\mu_{r,G_1,f}\in Q_{r, G_1, k}$ be an arbitrary element corresponding to a coloring $f: [i]\rightarrow [k]$. It is enough to prove that the 
total variation distance of $\mu_{r,G_1,f}$ and $\mu_{r,G_2,f}$ can be bounded from above by a quantity depending only on $i$ and tending to zero as $i$ goes to $\infty$. Let $e$ be the only edge in $E(G_1)\triangle E(G_2)$. In both $G_1$ and $G_2$ there are boundedly many vertices $v$ such that $e$ intersects the neighborhood of radius $r$ of $v$. It is easy to see that $2(d+1)^r$ is such a bound. The colored neighborhoods of the rest of the vertices are the   same in $G_1$ and $G_2$. It follows that the total variation distance of $\mu_{r,G_1,f}$ and $\mu_{r,G_2,f}$ is at most $2(d+1)^r/i$. This completes the proof. \hfill $\square$
\end{proof}

\subsubsection{Typical processes} 

In this section we prove  a correspondence principle between typical processes and random $d$-regular graphs. 

Throughout this section, $d\geq 3$ will be fixed, and  $\mathbb G_n$ will be a uniformly chosen random $d$-regular graph on $n$ vertices. 

\begin{lemma}\label{typl1} For fixed $d\geq 3$ there is a sequence  $\lbrace B_n\rbrace_{n=1}^{\infty}$ of $d$-regular graphs with $|V(B_n)|=n$ such that $d_{CN}(B_n, \mathbb G_n)$ tends to $0$ in probability as $n\rightarrow \infty$.
\end{lemma}

\begin{proof}
Given $\varepsilon>0$, for all $n\geq 1$, by using Lemma \ref{lem:net}, we choose an $\varepsilon/4$-net $N_n$ of size at most $F(\varepsilon/4)$ in the set of $d$-regular graphs with $n$ vertices  with respect to the colored neighborhood metric. (We emphasize that the size of the net does not depend on the number of vertices of the graph.) For each $n$, let $B_{n, \varepsilon}\in N_n$ be a (deterministic) $d$-regular graph on vertices such that 
\begin{equation}\label{conce1}\mathbb P(d_{CN}(B_{n, \varepsilon}, \mathbb G_n)\leq \varepsilon/4)\geq \frac{1}{F(\varepsilon/4)},\end{equation} 
where $\mathbb G_n$ is a uniform random $d$-regular graph on $n$ vertices.   
Such a $B_{n, \varepsilon}$ must exist according to the definition  of the $\varepsilon/4$-net $N_n$. 

We define $f_{n, \varepsilon}(H_n)=d_{CN}(B_{n, \varepsilon}, H_n)$ for $d$-regular graphs $H_n$ on $n$ vertices. By Lemma \ref{lem:lip}, if $n\geq n_0$ with some fixed $n_0$, then $f_{n, \varepsilon}$ is a Lipschitz function with $\delta$. By well-known concentration inequalities (based on the exploration process and Azuma's inequality on martingales, see e.g. \cite[Chapter 7]{alon}, this implies the following. For all $\eta>0$ there exists $n_1=n_1(\eta)$ such that  
\begin{equation}\label{conce2}\mathbb P(|f_{n, \varepsilon}(\mathbb G_n)-\mathbb E(f_{n, \varepsilon}(\mathbb G_n))|>\eta)\leq \eta \qquad (n\geq n_1).\end{equation}
By choosing $0<\eta<\min(\varepsilon/4, 1/F(\varepsilon/4))$, inequalities \eqref{conce1} and \eqref{conce2} together imply $\mathbb E(f_{n, \varepsilon}(\mathbb G_n))\leq \varepsilon/2$ $(n\geq n_1)$. That is, since $f_{n, \varepsilon}$ is concentrated around its expectation (due to its Lipschitz property) for large $n$, and $\mathbb G_n$ is close to some fixed graph with probability with a positive lower bound not depending on $n$, we conclude that this expectation has to be small for $n$ large enough. 

Putting this together, this yields 
\[\mathbb P(f_{n, \varepsilon}(\mathbb G_n)>\varepsilon)=\mathbb P(d_{CN}(B_{{n, \varepsilon}},\mathbb G_n)>\varepsilon)\leq \varepsilon \qquad (n\geq n(\varepsilon)).\]

By a standard diagonalization argument, let $k(n)=\max\{k\, \vert \, n(1/k)<n\}$ and $B_n=B_{n,1/k(n)}$. It is clear by the last inequality that $\{ B_n\}_{n=1}^{\infty}$ satisfies the requirement. \hfill $\square$

\end{proof}

\begin{proposition} \label{prop:graphing}For all infinite $S\subseteq \mathbb N$ there exists an infinite $S'\subseteq S$ and a graphing $\mathcal G\in \overline{X_d}$ such that if $(\mathbb G_i)_{i\in S'}$ is a sequence of independent $d$-regular random graphs with $|V(\mathbb G_i)|=i$, then $(\mathbb G_i)_{i\in S'}$ locally-globally converges to the 
graphing $\mathcal G$ with probability $1$.

\end{proposition}

\begin{proof} First, based on Lemma \ref{typl1}, we can choose $S_1\subseteq S$ such that $\{d_{CN}(B_n, \mathbb G_n)\}_{n\in S_1}$ tends to 0 with probability 1. On the other hand, by compactness, there is an infinite subsequence $S'\subseteq S_1$ such that $\{ B_n\}_{n\in S'}$ is locally-globally convergent. Let $\mathcal G$ be its limit. This completes the proof.
 \hfill $\square$

\end{proof}

\medskip

Graphings arising as the local-global limits of sequences of random graphs -- like in Proposition \ref{prop:graphing} -- play an important role when we are dealing with random $d$-regular graphs and typical processes. 

\begin{definition}
A graphing $\mathcal G\in \overline{X_d}$ is called typical if there exists an infinite $S'\subseteq \mathbb N$ such that if $\{\mathbb G_i\}_{i\in S'}$ is a sequence of independent $d$-regular random graphs with $|V(\mathbb G_i)|=i$, then $\{\mathbb G_i\}_{i\in S'}$  locally-globally converges to $\mathcal G$ with probability $1$.
\end{definition} 

We conjecture  that (with respect to local-global equivalence) there is a unique typical graphing. To put it in another way, the almost sure limit of sequences of random regular graphs does not depend on the sequence of the number of vertices. More precisely, the conjecture is the following. If $\mathcal G$ and $\mathcal G'$ are both typical graphings, then $\mathcal G$ and $\mathcal G'$ are locally-globally equivalent (i.e. their 
local-global distance is 0). This is essentially saying that a growing sequence of random $d$-regular graphs is convergent in probability. Deep results in favour of this conjecture were established by Bayati, Gamarnik and Tetali \cite{bayati}. They proved the convergence in probability of various graph parameters, e.g. the independence ratio. Note that the paper \cite{HLSz} has a formally stronger conjecture, which states convergence with probability 1. 

\medskip

We will need the following fact, which would also trivially follow from this conjecture. 

\begin{lemma} \label{lem:closed} The set of typical graphings is closed within the local-global topology in $\overline {X_d}$.
\end{lemma} 
\begin{proof}Let $\{\mathcal G_n\}_{n=1}^{\infty}$ be a sequence of typical graphings converging locally-globally to $\mathcal G$. We can assume that $\sum_{n=1}^{\infty} d_{CN}(\mathcal G_n, \mathcal G)$ is finite. By definition, for every $i\in \mathbb N$ there is an infinite set $S_i$ such that $\{\mathbb G_n\}_{n\in S_i}$ converges to $\mathcal G$ with probability 1. Choose $j_i\in S_i$ such that \[\sum_{i=1}^{\infty}\mathbb E(d_{CN}(\mathbb G_{j_i}, \mathcal G_i))<\infty.\] 
Using triangle inequality and our assumption on the seqence $\{\mathcal G_i\}$, we may replace $\mathcal G_i$ by $\mathcal G$, and the sum remains finite. This shows that the sequence of independent random graphs $\{\mathbb G_i\}_{i=1}^{\infty}$ locally-globally converges to $\mathcal G$ with probability 1, and hence $\mathcal G$ is a typical graphing. \hfill $\square$ 

\end{proof}

\medskip

Our goal is to understand the consequences of results on typical processes for random $d$-regular graphs. In order to do this, we recall that there is a connection between $d$-regular  graphings and invariant processes on the $d$-regular tree \cite{cordec, HLSz}, with the property that typical graphings correspond to typical processes. Suppose that $\mathcal G$ is a  $d$-regular graphing. Moreover, suppose that the vertices of $\mathcal G$ are colored with a finite color set $S$ in a measurable way. Then choose a random vertex of $G$ and 
map the rooted $d$-regular tree into $\mathcal G$  by a random graph covering such that the root is mapped to the chosen vertex. By assigning to each vertex of the $d$-regular tree the color of its image in $\mathcal G$, we get a random coloring of $T_d^*$. This way we get a random invariant process on $T_d^*$. Now we consider all the processes that can be obtained from $\mathcal G$ with an $S$-coloring. We denote by $\gamma(\mathcal G, S)$ the closure of this set in the weak topology. Note that $\gamma(\mathcal G, S)$ is invariant with respect to local-global equivalence of graphings. 

It follows immediately from the definition that if the graphing $\mathcal G$ is typical and $S$ is an arbitrary finite set, then all processes in $\gamma(\mathcal G, S)$ are typical. Furthermore, every typical process can be obtained this way.   By Lemma 
\ref{lem:closed} we get the next corollary. 

\begin{lemma}\label{lem:closedpr}For every fixed $d$ and finite set $S$, the set of typical processes with values in $S$ is closed with respect to the weak topology.
\end{lemma}

Now we are ready to prove the following correspondence principle between random graphs and typical graphings. 

\begin{proposition} \label{prop:corres}Let $(\mathbb G_i)_{i \in \mathbb N}$ be a sequence of independent random $d$-regular graphs with the number of vertices tending to infinity. Let $C$ be a closed subset of $\overline {X_d}$ with respect to the local-global topology. Suppose that $C$ does not contain any typical graphings. Then $P(\mathbb G_i\in C)\rightarrow 0$ as 
$i\rightarrow\infty$. 
\end{proposition}

\begin{proof} Assume that $S=\lbrace i\in \mathbb N: P(\mathbb G_i\in C)>\varepsilon\rbrace$ is infinite for some $\varepsilon>0$. Choose $S'\subseteq S$ by Proposition \ref{prop:graphing}; that is, $(\mathbb G_i)_{i\in S'}$ locally-globally converges to a fixed graphing $\mathcal G$ with probability 1. On the other hand, by independence, it follows that with probability 1 we have $\mathbb G_i\in C$ for infinitely many $i\in S'$. 
Since $C$ is closed in the local-global topology, and $\mathcal G$ is the limit of the whole sequence almost surely, this implies that $\mathcal G$ has to be in $C$. But, by definition, $\mathcal G$ is typical. This contradicts our assumption on $C$. \hfill $\square$
\end{proof}

\medskip

The main application of Proposition \ref{prop:corres} is that we can turn statements about typical processes into statements about random $d$-regular graphs. 
As we have explained before,  typical processes are exactly the processes coming from typical graphings. Therefore if we succeed in excluding typical processes from a closed set within the weak topology of invariant processes, then at the same time we exclude typical graphings from a closed set within the local-global topology, and through Proposition \ref{prop:corres} we obtain a result for random $d$-regular graphs. 
We will demonstrate this principle on concrete examples  in Section \ref{dominating}.

\subsection{Joinings and related metric}

\label{joining}

An invariant coupling, or shortly {\it joining}, of two elements $\mu,\nu\in I_d(M)$ is a process $\psi\in I_d(M\times M)$ such that the two marginal processes of $\psi$ (with respect to the first and second coordinate in $M\times M$) are  $\mu$ and $\nu$.
We denote by $C(\mu,\nu)$ the set of all joinings of $\mu$ and $\nu$.

Assume that the topology on $M$ is given by a metric $m:M\times M\rightarrow\mathbb{R}^+\cup\{0\}$. Then we define a distance $m_c$ on $I_d(M)$ in the following way.
\begin{equation}m_c(\mu,\nu)=\inf_{\psi\in C(\mu,\nu)}\mathbb{E}(m(\psi|_v)),\label{eq:metric}\end{equation}   
where $v$ is an arbitrary fixed vertex of $T_d$ and $\psi|_v$ is the restriction of $\psi$ to $v$. Note that automorphism invariance implies that $m_c$ does not depend on the choice of $v$.  
If $M$ has finite diameter, then $m_c(\mu,\nu)$ is a finite number bounded by this diameter. 

This is basically Ornstein's $\bar d$-metric, which was originally defined for $\mathbb Z$-invariant processes, see e.g. \cite{glasner}. See also the recent papers of Lyons and Thom \cite{russ, monoton1}  where 
several results and open questions on $T_d$ are presented,  connecting  the factor of i.i.d. processes to 
this metric.

 The key to the proof of the fact that this is a metric is the notion of relatively independent joining \cite[Chapter 15, Section 7]{glasner}. Assume that $\psi_{1,2}\in C(\mu_1,\mu_2)$ and $\psi_{2,3}\in C(\mu_2,\mu_3)$. Let us consider the unique joining of $\psi_{1,2}$ and $\psi_{2,3}$ that identifies the marginal $\mu_2$ and has the property that $\mu_1$ and $\mu_3$ are conditionally independent with respect to $\mu_2$. 
We remark that using relatively independent joinings and some kind of Borel--Cantelli arguments  one can check that the space of invariant processes is complete with respect to the $\bar d$-metric.

The case when $M$ is a finite set plays a special role in our paper. In this case we define 
$m(x,y)=1$ if $x\neq y$ and $m(x,x)=0$ for $x,y\in M$. The corresponding metric $m_c$ is regarded as the Hamming distance for processes in $I_d(M)$.

\medskip

\section{Glauber dynamics and branching Markov processes}

\label{glauber}

Glauber dynamics is an important tool in statistical physics. In this chapter we consider a variant of 
heat-bath Glauber dynamics that is an $m_c$-continuous transformation on $I_d(M)$. 
We begin with the finite case, then we define the Dobrushin coefficient, and formulate the main results: a Dobrushin-type sufficient condition for branching Markov chains to be factor of i.i.d. 
Then we give a brief description of the Poisson Glauber dynamics that seems to be the closest analogy to classical Glauber dynamics, and we define something similar, that is more technical, but more useful in our applications. 

\subsection{Glauber dynamics on finite graphs}

\label{finiteglaub}

First suppose that $G$ is a (potentially infinite) $d$-regular graph, and we have a reversible Markov chain with finite state space $S$ 
and transition matrix $Q$. We think of $G$ such that each vertex has a state from $S$; the state of the graph is an element in $S^{V(G)}$. A {\it Glauber step  at vertex} $v\in V(G)$ is a way of generating a random state from a given state of the graph. 
We do this by randomizing the state of $v$ conditionally on the states of its neighbors, as follows. 

Let $N(v)$ denote the set of the neighbors of $v$. Let $C=v\cup  N(v)$ and $\mu_C$ the distribution of the branching Markov process restricted to $C$. For a state $\omega\in S^{N(v)}$, we define $B_{v, \omega}$ to be the conditional distribution of the state of $v$ given  $\omega$. The Glauber step at $v$ (the so called heat-bath version) is the operation of randomizing the state of $v$ from $B_{v, \omega}$.

Now we define the Glauber dynamics on a finite graph. It is a Markov chain on the state space of the graph $S^{V(G)}$ obtained by choosing a vertex $v$ uniformly at random, and performing the Glauber step at $v$. 
See e.g. Section 3.3. in \cite{markovmixing} on Glauber dynamics for various models.

 It is also clear from the theory of 
finite state space Markov chains that (with appropriate conditions on $Q$) this Markov chain has a unique stationary 
distribution, which is the limiting distribution of the Glauber dynamics. However, the order of the mixing time depends on $Q$; the question typically is whether the mixing 
time can be bounded by a linear 
function of the number of vertices. Our main result will show that the so called Dobrushin condition, which implies fast mixing, also implies that the process is factor of i.i.d. Note that the connection between fast mixing and factor of i.i.d. property was also implicitly used in \cite{gamarnik}. A paper of Berger, Kenyon, Mossel and Peres \cite{berger} deals with the problem of fast mixing on trees for the Ising model, i.e. when there are only two states. See Theorem 1.4. of \cite{berger}. Furthermore Mossel and Sly \cite{exact} gave a sharp threshold for 
general bounded degree graphs. The recent paper 
of Lubetzky and Sly \cite{spacetime} contains more refined results for the Ising model with underlying graph 
$(\mathbb Z/n\mathbb Z)^d$, and its Theorem 4 refers to analogous results for general graphs.  

It is important to mention the paper of Bubley and Dyer \cite{pathcoupling} on   fast mixing of 
the Glauber dynamics of Markov chains and on the path coupling technique, which 
is applied in  \cite{berger}, 
and whose ideas will be used in what follows. 
See also the paper of Dembo and Montanari \cite{dembo} and Chapter 15 in \cite{markovmixing} for more details on mixing time of the Glauber dynamics.

\subsection{The Dobrushin coefficient and factor of i.id. processes} 

When we examine how the properties of the Glauber dynamics depend on the transition matrix $Q$, it is helpful to investigate the following: how does a change in the state of a single neighbor of $v$ effect the conditional 
distribution of the state of $v$ at the Glauber step? This is the idea of the definition of the Dobrushin coefficient (see e.g. 
\cite{pathcoupling, dobrushin}). 

\begin{definition}[Dobrushin coefficient] \label{def:dobr}Let us consider a reversible Markov chain on a finite state space $S$ with transition matrix $Q$. 
The Dobrushin coefficient of the Markov chain is defined by 
\begin{multline*}D=\sup \bigl \lbrace d_{TV}( B_{v,\omega}, B_{v,\omega'}): \omega, \omega'\in  S^{N(v)},\ |\lbrace u\in N(v): \omega(u)\neq \omega'(u)\rbrace|=1 \bigr \rbrace, \end{multline*} 
where  $d_{TV}$ is the total variation distance of 
probability distributions:
\begin{multline*}d_{TV}(P_1,P_2)
=\frac{1}{2}\sum_{s\in S}|P_1(s)-P_2(s)|\\=\inf\lbrace \mathbb P(X\neq Y): X\sim P_1,\ Y\sim P_2,\ \mathbb P \textrm{\ is a coupling of } X \textrm{\ and\  }Y
\rbrace.\end{multline*}
\end{definition}

To put it in another way, we consider pairs of configurations on the neighbours of $v$ that differ at only one place. 
We calculate the total variation distance of the conditional distributions at $v$ given the two configurations. 
Finally we take the supremum for all these pairs. Note that this definition depends only on $Q$ and 
the number of neighbors of $v$.

\medskip

Now we can formulate the main result of this section, which will be proved in Subsection \ref{proofthm1}. 

\begin{theorem}\label{thm1}
If the condition $D<1/d$ holds for a reversible Markov chain with transition matrix $Q$ on a finite state space $S$, then the 
branching Markov process $\nu_Q$ corresponding to $Q$ on the $d$-regular tree $T_d$ is a factor of i.i.d. process; that is, $\nu_Q\in F_d(S)$.
\end{theorem}

This theorem is 
heuristically in accordance with the results of Bubley and Dyer \cite{pathcoupling}, who proved fast mixing of the Glauber 
dynamics if the condition $D<1/d$ holds.  Moveover, this condition has other consequences for correlation decay and the 
uniqueness of the Gibbs measure under various circumstances \cite{dobrushin, lovasz, sokal, weitz}. However, we do not know in general 
whether fast mixing or the uniqueness of the Gibbs measure implies that the branching Markov process is  factor of i.i.d.

\subsection{Poisson Glauber dynamics on $T_d$}

When the vertex set of the underlying graph is finite, as we have already seen in Subsection \ref{finiteglaub}, it is easy to define the Glauber dynamics. 
From now on we get back to the infinite $d$-regular tree, where it is not possible to choose a vertex uniformly at 
random, and perform Glauber dynamics step by step this way. 
In this subsection we give a heuristic description of the continuous time Glauber dynamics on the infinite tree for motivation. However, for our purposes the discrete version defined in the next subsection is 
more convenient, hence we omit the precise details of the definition of the continuous time model.

We assign independent Poisson processes with rate 1 to the vertices of the tree. That is, each vertex has a 
sequence of random times when it wakes up. At the beginning, at time zero, the vertices are in random 
states chosen independently and uniformly from the finite state space $S$. When a vertex wakes up, it performs a single 
Glauber step defined earlier. This depends only on the state of the neighbors of the 
vertex. However, to know these states, we have to know what has happened when the neighbors have performed Glauber steps earlier. 
This continues, hence it is not trivial whether this process is well-defined. To see this, one can check that the 
expectation of the number of Glauber steps that effect the randomization of a vertex waking up is finite. 

This argument could be made precise (see e.g. \cite[Theorem 1]{howard} for the definition of joint distribution of the Poisson processes on $T_3$). The advantage of the continuous time Glauber dynamics is the fact that 
the probability that neighbors wake up at the same time is zero. When we define the discrete time Glauber step 
in the next subsection, we will have to pay attention to avoid the event that neighbors are waking up simultaneously.

\subsection{The factor of i.i.d. Glauber step on $T_d$}

As we have seen in Subsection \ref{finiteglaub}, the single Glauber step for finite graphs maps each configuration in 
$S^{V(G)}$ to a random configuration. Now we are working with the infinite $d$-regular tree $T_d$, hence 
we deal with random processes, which are probability distributions on $S^{V(T_d)}$. We 
will describe a way of performing Glauber steps simultaneosly at different vertices such that our procedure produces factor of i.i.d. 
processes from factor of i.i.d. processes. 

Given a configuration  $\omega \in S^{V(T_d)}$, which is a labelling of 
the vertices of the $d$-regular tree with labels from the finite state space $S$ of the Markov chain, we will perform a 
single Glauber step to get a random configuration $G\omega$ in $S^{V(T_d)}$. Fix the 
transition matrix $Q$. The scheme is 
the following; we give the details afterwards. 
\begin{enumerate}
\item Choose an invariant random subset $U$ of $V(T_d)$ such that it has positive density  and it does not contain any two vertices of distance less than 3. 
\item For each vertex $v\in U$ perform the usual Glauber step at $v$: randomize the state of vertex $v$ according to 
the conditional distribution with respect to the states of its neighbours.
\end{enumerate}

More precisely, for the first part we need the following lemma. 

\begin{lemma} \label{ebred}It is possible to find an invariant random subset $U$ 
of $V(T_d)$ such that \begin{itemize}
\item it is factor of i.i.d.: the distribution of the indicator function of $U$ is in $F_d(\lbrace 0,1\rbrace)$;
\item it has positive density: the probability 
that the root $o$ is in $U$ is positive;
\item it does not contain any two vertices of distance less than 3. 
\end{itemize}
\end{lemma}

\begin{proof} We start with $[0,1]^{V(T_d)}$ endowed with $\mu$, the product measure of the uniform distributions on the 
interval $[0,1]$. That is, vertices have independent and uniformly distributed labels from $[0,1]$. 

A vertex $v\in V(T_d)$ will be in $U$ if its label is larger than the labels of the vertices in its neighbourhood of 
radius 2. That is, for $\omega\in [0,1]^{V(T_d)}$ we set $f(\omega)=1$ if $\omega$ at the root $o$ is larger than $\omega_u$ for all $u\in V(T_d)$  at distance at most 2 from the root. Otherwise $f(\omega)=0$. Then we get the characteristic function of $U$ by placing the root to each vertex and applying $f$. This is a factor of i.i.d. process satisfying all conditions. \hfill $\square$

\end{proof}

\medskip

This lemma ensures that we can perform the first part of the Glauber step as a factor of i.i.d. process. 
As for the second part, we just refer to the definition of the Glauber step at a single vertex: each vertex $v\in U$ randomizes its 
state given the state of its neighbors and according to the distribution of the branching Markov process 
constrained on the finite subset $v\cup N(v)$. Since the distance of any two vertices 
in $U$ is at least 3, these randomizations can be performed simoultaneously and independently.

It is straightforward to  extend the definition of the Glauber step to a map from the set of probability measures on $S^{V(T_d)}$ to 
itself. Namely, choose a random configuration from $S^{V(T_d)}$ according to the given measure, and perform the 
Glauber step described above. This gives a new probability measure on $S^{V(T_d)}$. It is also easy to see 
that if we apply this for an invariant probability measure, then the resulting measure will also be invariant. 
Hence we have extended the definition of the Glauber step to a transformation of the form $G: I_d(S)\rightarrow I_d(S)$.

Moreover, note that if  $\nu$ is factor of i.i.d., then $G(\nu)$ is also factor of i.i.d., 
since the set of vertices performing Glauber steps is chosen by a factor of i.i.d. process by Lemma \ref{ebred}, and 
Glauber steps depend only on the state of the neighbors of these vertices.

\subsection{The invariance of the branching Markov process for the Glauber step} 

In order to prove Theorem \ref{thm1}, we will need the fact that the Glauber step defined above does not change 
the distribution of the branching Markov process. 

\begin{proposition}[Invariance]  \label{prop:inv}
If $\nu_Q\in I_d(S)$ is the branching Markov process with transition matrix $Q$
then it is a fixed point of the Glauber step corresponding to $Q$ and $d$ (i.e. $G(\nu_Q)=\nu_Q$.)
\end{proposition}

\begin{proof}
First we check that the Glauber step at a single vertex  $u$ does not change the 
distribution of the branching Markov process. It follows from the fact that the distribution of the state of $u$ and the joint distribution of the states at $V(T_d)\setminus \{u\cup N(u)\}$ are conditionally independent given the states of the vertices in $N(u)$.  

Let $U$ be the set of vertices performing Glauber steps when we apply $G$. Since these vertices are far away from each other 
(their distance is at least 3 according to Lemma \ref{ebred}), the randomizations are independent, and therefore, since the Glauber step at a 
single vertex does not change the distribution, it is also invariant for finitely many steps. On the other hand, for arbitrary $U$ it is possible to 
find finite sets of vertices $U_n$ such that (i) $U_n\subseteq U_{n+1}$ for all $n$; (ii) $\bigcup_{n=1}^{\infty} U_n=V(T_d)$; (iii) if a vertex is in $U\cap U_n$, then all its neighbors are in $U_n$. For example, one can use balls of appropriate radius with a few vertices 
omitted from the boundary. Since every $U_n$ contains finitely many vertices, and vertices on the boundary of $U_n$ do not perform Glauber steps, the distribution of the branching Markov process is invariant for the 
Glauber steps at vertices $U\cap U_n$. This also implies that the branching Markov process is invariant for $G$, when we perform Glauber steps at 
the vertices of $U$ simultaneously.
\hfill $\square$
\end{proof}

\subsection{The Glauber step as a contraction}

We will prove that if the Dobrushin coefficient (Definition \ref{def:dobr}) is small enough, then the factor of i.i.d. Glauber step is a contraction 
with respect to the metric 
$m_c$ derived from the Hamming distance on $S$.
First we need a notation and a lemma. 

\begin{definition}[Coupling Hamming distance] Let $S$ be a finite state space with the discrete topology and with the Hamming distance: $m(s,s)=0$ for all $s\in S$ and $m(s,t)=1$ if $s\neq t$. We denote by $h_c$ the metric defined by equation \eqref{eq:metric} on $I_d(S)$ corresponding to the Hamming distance (see Section \ref{joining}). 
\end{definition}

Recall that $B_{v,\omega}$ is the distribution of the state of vertex $v$ at the Glauber step if the state of its 
neighbors are given by $\omega\in S^{N(v)}$.

\begin{lemma}\label{lem:pc}Suppose that we have a branching Markov process on $T_d$ with Dobrushin coefficient $D$. Fix 
$v\in V(T_d)$ and $\omega, \omega'\in S^{N(v)}$ such that  $|\lbrace u\in N(v): \omega(u)\neq \omega'(u)\rbrace|=k$. Then we have that 
\[d_{TV}(B_{v,\omega}, B_{v, \omega'})\leq k D.\]

\end{lemma}
\begin{proof} The case $k=1$ is trivial. The general case follows by induction using the triangle inequality. \hfill $\square$
\end{proof}

\medskip

Now we can prove that the factor of i.i.d. Glauber step is a contraction if the Dobrushin condition holds.
\begin{proposition}\label{prop:contract}
If $D<{1/d}$, then $G: I_d(S)\rightarrow I_d(S)$ is a contraction with respect to the coupling Hamming distance $h_c$; that is, 
there exists $r<1$ such that
\[h_c(G(\nu_1), G(\nu_2))<r\cdot h_c(\nu_1, \nu_2)\qquad \text{ for all } \nu_1, \nu_2\in I_d(S).\]
\end{proposition}

\begin{proof} Choose  $\varepsilon>0$ such that $r:=(1+\varepsilon)(1-p+pdD)<1$, where $p>0$ is the density of $U$ in the Glauber step. This is possible if $D<1/d$. Fix $\nu_1, \nu_2\in I_d(S)$. Denote their distance $h_c(\nu_1, \nu_2)$  by $h$. By the definition of the metric $h_c$, there is 
a  joining $\Psi$ of $\nu_1$ and $\nu_2$ such that  $\mathbb E(m(\Psi|_v))<(1+\varepsilon)h$ holds ayt any given vertex $v$, where $m$ denotes the Hamming distance on $S$. 

Our goal is to construct a joining $\Psi'$ of $G(\nu_1)$ and $G(\nu_2)$ such that $\mathbb{E}(m(\Psi'|_v))\leq rh$.
We construct this joining in a way that the set of vertices that perform the Glauber step are the same for $\nu_1$ and $\nu_2$. 
As a first step we choose an invariant random set $U$ according to Lemma \ref{ebred} such that $U$ is independent from $\Psi$. 

We define $\Psi'$ from $\Psi$ and $U$ as follows. When we randomize the state of a given vertex $v\in U$, conditionally on the states of vertices in $N(v)$, we use the best possible coupling of the conditional distributions in total variation (the probability that the two random variables are different is minimal). Since we deal with finite number of configurations and a discrete probability space for fixed $u$, this is sensible. For the distinct vertices in $U$ we join these couplings independently to get $\Psi'$ for a fixed $U$. This defines $\Psi'$ on the whole extended probability space. 

Since $U$ is invariant and the randomizations depend only on the states of the neighbors, $\Psi'$ is also invariant. 
It is clear that the marginal distributions $\nu_1'$ and $\nu_2'$ of $\Psi'$ are identical to $G(\nu_1)$ 
and $G(\nu_2)$, respectively. 

Now we give an upper bound on the coupling Hamming distance of $\nu_1'$ and $\nu_2'$. 

Fix $v\in V(T_d)$. The probability that $v\in U$ is  $p$ by definition. With probability $1-p$ its state is not changed, therefore there is a difference in $\Psi'$ with probability $E(m(\Psi|_v))<h(1+\varepsilon)$; this is 
just the density of differences in the original process. Otherwise a Glauber step is performed at $v$. The expected value of the number of differences in $N(v)$ between the random configurations according to $\nu_1$ and $\nu_2$ is $dE(m(\Psi|_u))<dh(1+\varepsilon)$. By Lemma \ref{lem:pc}, if the number of 
differences is $k$, then it is possible to couple the conditional distributions such that the probability that 
the state of $v$ is a difference is less than or equal to $kD$. When we defined $\Psi'$, we have chosen the 
best couplings with respect to total variation. Therefore the probability that we see a difference in $\Psi'$ is less than 
$(1-p)h(1+\varepsilon)+pdhD(1+\varepsilon).$ By the choice of $\varepsilon$ (where we used the 
condition $D<1/d$) this is less than  $h$, and we get that 
\[h_c(G(\nu_1), G(\nu_2))<(1-p)h(1+\varepsilon)+pdhD(1+\varepsilon)=rh.\]
 \hfill $\square$

\end{proof}

\medskip
Now, putting Proposition \ref{prop:inv} and Proposition \ref{prop:contract} together one can easily  show that the branching Markov process is 
a limit of factor of i.i.d. (it belongs to $L_d(S)$) with respect to the $\bar d$-metric if the Dobrushin coefficient is smaller than $1/d$. 

Namely, we start with an i.i.d. labelling of the vertices of the tree by labels from $S$; this is measure $\nu_0$. We have checked that if a given invariant process is factor of 
i.i.d., then its image under the Glauber step $G$ is also factor of i.i.d. Therefore if we apply $G$ finitely many times, we also get a factor of i.i.d. process. 
By Proposition \ref{prop:inv} the branching Markov process is a fixed point of $G$. A contraction can not have more than one fixed points, and hence it is also clear that $G^n \nu_0$ (which is a factor of i.i.d. process) converges to the branching Markov process in the $\bar d$-metric exponentially fast.  

However, in the next section we will prove the stronger statement that the branching Markov process is itself a factor of i.i.d. process if $D<1/d$ holds.

\subsection{Proof of Theorem \ref{thm1}}

\label{proofthm1}

  Recall that the Glauber step $G$ can be 
 defined as a map from the set of invariant processes to itself. It   is 
 a contraction with respect to the $\bar d$-metric, whose unique fixed point is the corresponding branching Markov process. Moreover, it maps factor of i.i.d. processes to factor of i.i.d. processes.

\begin{proof} First we define an operation $T$ on sequences of processes that are already coupled to each other somehow. 
More precisely, let $(J_1, J_2, \ldots)$ be a (possibly infinite) sequence of invariant processes from $I_d(S)$ defined on the same 
probability space.  Then $^T(J_1, J_2, \ldots)$ will also be a sequence of invariant processes. The distribution of 
the $k$th term of   $^T(J_1, J_2, \ldots)$ will be identical to the distribution of $G(J_k)$. The main point is the 
coupling of these processes. First we couple $G(J_1)$ and $G(J_2)$ such that, at each vertex where a Glauber step is performed, the coupling realizes the total 
variation distance of the conditional distributions given the states of the neighbors.  Then we couple $G(J_3)$ to the already existing probability space such that it is optimally coupled to $G(J_2)$ with respect to the total variation distance. 
We continue this, from the 
left to the right, we always couple the next term to the previous one with the coupling that realizes the total variation distance at each vertex.

Let $I$ be the i.i.d. process on ${T_d}$ whose marginal distributions at the vertices are uniform on $S$. We define
\[I^{(n)}=\,^T\!(I,\, ^T\!(I,\, ^T\!(I, \ldots, \,^T\!(I,\, ^T\!(I))))),\] 
with $n$ copies of $I$ as follows. We already know that $T$ maps any sequence of invariant processes to 
another sequence of processes of the same length. When we have a sequence, and we write an $I$ before it, 
we mean the sequence consisting of a copy of $I$ and the original sequence coupled to a common probability 
space independently. We get a longer sequence, and we apply $T$ to this. Then again, we add an independent copy of $I$, and apply $T$. We repeat this $n$ times to get $I^{(n)}$.
It is also clear that the $k$th term of this sequence of length $n$ is identical in distribution to $G^k I$. Therefore it belongs to $F_d(S)$.

When we are producing this sequence, we are using the following probability spaces that are coupled to each other. 
First, we need the spaces where these copies of $I$ are defined. Then, when we apply the Glauber step, we need to choose the random set of vertices waking up, like in Lemma \ref{ebred}. Finally, there are the moves when the given vertices randomize their current state with the appropriate coupling. 

The next step is to show that $I^{(\infty)}$ also makes sense. It will have infinitely many coordinates. 
Since we performed the coupling procedure from the left to the right, if we want to determine the $k$th term of 
$I^{(\infty)}$, then it is sufficient to deal with the first $k$ copies of $I$ and choose the optimal couplings defined above 
finitely many times. Hence the whole sequence is  well defined.
 
We go further and we will see that $I^{(\infty)}=(H_1, H_2, \ldots)$ is a factor of i.i.d. process. 
In the construction of $I^{\infty}$ we use the following independent random variables, uniformly distributed on $[0,1]$: 
\begin{enumerate}
\item for each application of $T$, we need the random set from Lemma \ref{ebred}; this requires an independent copy of $[0,1]$ associated with each vertex of $T_d$; 
\item for each application of $T$, we need countably many copies of $[0,1]$ associated with each vertex to perform the Glauber steps and their couplings.
\end{enumerate}
It is easy to see that each coordinate of $I^{\infty}$ depends measurably on finitely many of these random variables. It follows that $I^{\infty}$ is factor of i.i.d.

We claim that for each vertex $v\in V(T_d)$ 
the sequence $(H_k(v))$ is constant except for finitely many terms almost surely, and the process $\nu$ defined by $v\mapsto \lim_{k\rightarrow\infty} H_k(v)$ is 
a factor of i.i.d. process.  Let $p_k$ be the probability that the root has a different state in $H_k$ and $H_{k+1}$. Since the Glauber step is a contraction with 
respect to the Hamming distance, $p_k$ tends to 0 exponentially fast. A Borel--Cantelli argument implies that the state of a given vertex stabilizes after 
finitely many steps. Then it follows that the limit is measurable and so it is factor of i.i.d.

Finally, since $H_k$ converges to the fixed point of $G$, which is the 
branching Markov process by Proposition \ref{prop:inv}, we get that the branching Markov process is factor of i.i.d. \hfill $\square$
\end{proof}

\section{Entropy inequalities}

\label{entropy}

In this section we will formulate necessary conditions for invariant processes to be typical based on 
entropy. These inequalities imply necessary conditions for a process to be factor of i.i.d. Note that these kind of inequalities were used for various purposes. They are closely related to the results of Bowen \cite{lewis} on $f$-invariant for factors of shifts on free groups (e.g. for the factor of i.i.d. case when $d$ is even).  Rahman and Vir\'ag also use this tool for examining
independent sets in factor of i.i.d. processes on the $d$-regular tree; see Section 2 of \cite{mustazee}.

Now we define configuration entropy as we will use it later on. Recall that if $\mu$ is a 
probability distribution on a finite set $S$ of atoms with probabilities $p_1,p_2,\dots,p_K$, then its entropy is defined by $h(\mu)=-\sum_{i=1}^K p_i \ln p_i$. (If a probability $p_i$ is zero, then the corresponding term is also defined to be equal to zero.) We also define $H(\mu):=e^{h(\mu)}$. Assume that a finite set of size $n$ has an $S$-coloring with color distribution $\mu$. Let $H(\mu,n)$ denote the number of such colorings. Then $H(\mu,n)=H(\mu)^{n(1+o(1))}$ as $n$ tends to infinity.

\begin{definition}[Configuration entropy] Let $\nu\in I_d(S)$ be an invariant measure on $S$-valued processes on $T_d$, where 
$S$ is a finite set. Fix a finite set $F\subset T_d$. The measure $\nu$ induces a probability distribution on the $S$-colorings of $F$ (that is, on the finite set $S^{V(F)}$). Let the configuration entropy $h(F)$ be the entropy of this probability 
distribution. 
\end{definition} 

The invariance of $\nu$ implies that $h(F)=h(F')$ whenever there is an automorphism of $T_d$ taking $F$ to $F'$. This means that it makes sense to talk about the entropy of a given configuration in $T_d$ (for example an edge or a star) without specifying where the given configuration is in $T_d$.  

We prove two entropy inequalities, which hold for every typical process, and 
hence for every universal and factor of i.i.d. process by Lemma \ref{lem:bovul}.  

Recall from Section \ref{invproc} that $R_d$ is the set of invariant processes that 
can be modelled on random $d$-regular graphs. We denote by $h(\begin{picture}(4,10)
\put(2,0){\line(0,1){6}}
\put(2,0){\circle*{2}}
\put(2,6){\circle*{2}}
\end{picture})$ the edge entropy, that is, $h(F)$  
when the finite graph is an edge, and $h(\begin{picture}(4,10)
\put(2,3){\circle*{2}}
\end{picture})$ will be the vertex entropy, where $F$ is a single vertex.

\begin{theorem}\label{edgevertex} For any typical process $\nu\in R_d$ the following holds:
\[\frac d2 h(\begin{picture}(4,10)
\put(2,0){\line(0,1){6}}
\put(2,0){\circle*{2}}
\put(2,6){\circle*{2}}
\end{picture})\geq(d-1)h(\begin{picture}(4,10)
\put(2,3){\circle*{2}}
\end{picture}).\]
\end{theorem}

Before proving Theorem \ref{edgevertex} we need a lemma. Let $PM(k)$ denote the number of perfect matchings on a set with $k$ elements. 

\begin{lemma}\label{entlem} Let $V$ and $S$ be finite sets where $|V|=n$. Let $\mu$ be a probability distribution on $S$ and let $\nu$ be a probability distribution on $S\times S$. Assume that $f:V\rightarrow S$ is a coloring of $V$ such that the color of a random element in $V$ has distribution $\mu$. Let $M_f$ be the set of perfect matchings on $V$ such that the pair of colors on the two endpoints of a random directed edge in the matching has distribution $\nu$. Assume that $M_f$ is not empty. Then $|M_f|=PM(n)H(\nu,n/2)H(\mu,n)^{-1}$.
\end{lemma}

\begin{proof}
Let $M'=\cup_g M_g$ where $g$ runs through the $S$-colorings of $V$ with color distribution $\mu$. We compute $|M'|$ in two different ways. It is clear that $|M'|=H(\mu,n)|M_f|$. On the other hand we can generate an element in $M'$ by first choosing a perfect matching on $V$ and then putting colors on the endpoints of the edges in a way that the distribution of colored edges is $\nu$. This can be done in $PM(n)H(\nu,n/2)$ different ways. So we obtain that $H(\mu,n)|M_f|=PM(n)H(\nu,n/2)$. The proof is complete.
\end{proof}  

Now we are ready to prove Theorem \ref{edgevertex}.

\begin{proof}
The basic idea is the following. Assume that $S$ is a finite set and $\nu\in R_d$ is a typical process, which belongs to $I_d(S)$. We denote by $\{n_i\}_{i=1}^{\infty}$ the sequence such that $\nu\in [\{\mathbb G_{n_i}\}_{i=1}^\infty]$ holds with probability 1. Let  $\nu_v$ denote the marginal of $\nu$ on a vertex in $T_d$ and let $\nu_e$ denote the marginal of $\nu$ on an edge in $T_d$. Let $\varepsilon>0$. We denote by $G_{n,\varepsilon}$ the set of $S$-colored $d$-regular graphs on the vertex set $V_n$ with the restriction that the distribution of vertex colors is $\varepsilon$-close to $\nu_v$ and the distribution of colored (directed) edges is $\varepsilon$-close to $\nu_e$ in total variation distance.
Since $\nu$ is typical we know that if $n$ is large enough and belongs to the sequence $\{n_i\}_{i=1}^{\infty}$, then  almost every $d$-regular graph on $n$ vertices is in $G_{n,\varepsilon}$. It follows that 
\begin{equation}\label{entpr1}
\limsup_{n\to\infty} \frac{|G_{n,\varepsilon}|}{t_n}\geq 1
\end{equation}
 holds for every $\varepsilon>0$ where $t_n$ is the number of $d$-regular graphs on $n$ vertices.

In the rest of the proof we basically compute the asymptotic behavior of $\log|G_{n,\varepsilon}|$ if $\varepsilon$ is small and $n$ is large enough depending on $\varepsilon$. We start by assigning $d$ half-edges to each element of $V_n$. Let $V_n^*$ denote the set of these half edges. We first color the vertices according to the distribution $\nu_v$. We color $V_n^*$ such that each half edge inherits the color of its incident vertex. Then we match these half-edges 
such that the distribution of the colors of the endpoints of a uniform random edge is $\nu_e$.
To be more precise, in each coloring throughout this proof, we allow an $\varepsilon$ error in the total variation distance of distributions.  

There are $H(\begin{picture}(4,10)
\put(2,3){\circle*{2}}
\end{picture})^{n(1+o(1))}$ ways 
to color $V_n$ with distribution $\nu_v$. 
Here $o(1)$ means a quantity that goes to $0$ if first $n$ goes to infinity and then $\varepsilon$ goes to $0$.
 
Assume that the vertices of $V_n$ have a fix coloring. Let $M$ denote the set of perfect macthings on $V_n^*$ that satisfy the above requirement. By Lemma \ref{entlem} we have that $$|M|= PM(nd)H(\begin{picture}(4,10)
\put(2,0){\line(0,1){6}}
\put(2,0){\circle*{2}}
\put(2,6){\circle*{2}}
\end{picture})^{nd/2(1+o(1))}/H(\begin{picture}(4,10)
\put(2,3){\circle*{2}}
\end{picture})^{nd(1+o(1))}.$$

Finally we have to take into consideration that the order of the half-edges does not matter, hence we 
get every coloring $(d!)^{n}$ times. 

Putting everything together, the number of colored $d$-regular graphs on $V_n$ with the required property is the following:
\[\frac{H(\begin{picture}(4,10)
\put(2,3){\circle*{2}}
\end{picture})^{n(1+o(1))} PM(nd)H(\begin{picture}(4,10)
\put(2,0){\line(0,1){6}}
\put(2,0){\circle*{2}}
\put(2,6){\circle*{2}}
\end{picture})^{nd/2(1+o(1))}}{H(\begin{picture}(4,10)
\put(2,3){\circle*{2}}
\end{picture})^{nd(1+o(1))}(d!)^n}.\] 

Using the same argument about the half-edges but forgetting about all colorings, one can see that the number of 
$d$-regular graphs on $n$ vertices is 
\[\frac{PM(nd)}{(d!)^n}.\]

By (\ref{entpr1}) we conclude that
\[\limsup_{n\to\infty} \frac{H(\begin{picture}(4,10)
\put(2,3){\circle*{2}}
\end{picture})^{n(1+o(1))}H(\begin{picture}(4,10)
\put(2,0){\line(0,1){6}}
\put(2,0){\circle*{2}}
\put(2,6){\circle*{2}}
\end{picture})^{nd/2(1+o(1))}}{H(\begin{picture}(4,10)
\put(2,3){\circle*{2}}
\end{picture})^{nd(1+o(1))}}\geq 1;\]
\[ H(\begin{picture}(4,10)
\put(2,0){\line(0,1){6}}
\put(2,0){\circle*{2}}
\put(2,6){\circle*{2}}
\end{picture})^{d/2(1+o(1))}\geq H(\begin{picture}(4,10)
\put(2,3){\circle*{2}}
\end{picture})^{(d-1)(1+o(1))}.\]
 
By tending to $0$ with $\varepsilon$, taking the logarithm of both sides and rearranging we get the statement of the theorem. \hfill $\square$
\end{proof}

Similarly to the proof of Theorem \ref{edgevertex}, one can show the following.

\begin{theorem}\label{staredge} For any typical process $\nu\in R_d$ the following holds:
\[h(\begin{picture}(10,12)
\put(5,3){\circle*{2}}
\put(5,8){\circle*{2}}
\put(5,-2){\circle*{2}}
\put(0,3){\circle*{2}}
\put(10,3){\circle*{2}}
\put(5,3){\line(0,1){5}}
\put(5,-2){\line(0,1){5}}
\put(5,3){\line(1,0){5}}
\put(0,3){\line(1,0){4}}
\end{picture}_d)\geq \frac{d}{2} h(\begin{picture}(4,10)
\put(2,0){\line(0,1){6}}
\put(2,0){\circle*{2}}
\put(2,6){\circle*{2}}
\end{picture}),\]
where \  \begin{picture}(10,12)
\put(5,3){\circle*{2}}
\put(5,8){\circle*{2}}
\put(5,-2){\circle*{2}}
\put(0,3){\circle*{2}}
\put(10,3){\circle*{2}}
\put(5,3){\line(0,1){5}}
\put(5,-2){\line(0,1){5}}
\put(5,3){\line(1,0){5}}
\put(0,3){\line(1,0){4}}$\ \ \ \ _d$
\end{picture} \ \ is the star of degree $d$.
\end{theorem}

\begin{proof} The proof is very similar to the proof of Theorem \ref{edgevertex} so we only give the details that are different. Let $\nu\in R_d\cap I_d(S)$. Let $C$ denote the star of degree $d$. We label the root of $C$ by $0$ and the endpoints of the rays by $\{1,2,\dots,d\}$. Let \ $\begin{picture}(10,12)
\put(5,3){\circle*{2}}
\put(5,8){\circle*{2}}
\put(5,-2){\circle*{2}}
\put(0,3){\circle*{2}}
\put(10,3){\circle*{2}}
\put(5,3){\line(0,1){5}}
\put(5,-2){\line(0,1){5}}
\put(5,3){\line(1,0){5}}
\put(0,3){\line(1,0){4}}
\end{picture}_d$\ and\ $\begin{picture}(4,10)
\put(2,0){\line(0,1){6}}
\put(2,0){\circle*{2}}
\put(2,6){\circle*{2}}
\end{picture}$\ denote the marginal distributions of $\nu$ on the degree $d$ star and on an edge in $T_d$.  Again we count $S$-colored $d$-regular graphs on $n$ vertices with the restriction that the distribution on random stars and edges are close to\ $\begin{picture}(10,12)
\put(5,3){\circle*{2}}
\put(5,8){\circle*{2}}
\put(5,-2){\circle*{2}}
\put(0,3){\circle*{2}}
\put(10,3){\circle*{2}}
\put(5,3){\line(0,1){5}}
\put(5,-2){\line(0,1){5}}
\put(5,3){\line(1,0){5}}
\put(0,3){\line(1,0){4}}
\end{picture}_d$ and\ $\begin{picture}(4,10)
\put(2,0){\line(0,1){6}}
\put(2,0){\circle*{2}}
\put(2,6){\circle*{2}}
\end{picture}$.  Let $V_n$ be a set of $n$ elements. To each element $v_i\in V_n$ we assign $d$ half-edges $\{v_{i,j}\}_{j=1}^d$. We denote by $V_n^*$ the set of half-edges. Let $f:V_n^*\rightarrow S\times S$ be a coloring of the half-edges with pairs of elements from $S$ such that the first coordinates of $f(v_{i,j})$ and $f(v_{i,k})$ are the same, say $g(i)\in S$, for every triple $1\leq i\leq n$ and $1\leq j,k\leq d$. 
To each number $1\leq i\leq n$ we can assign an $S$-colored version of the star $C$ such that the color of the root $0$ is $s_i$ and the color of $j\in V(C)$ is the second coordinate of $f(v_{i,j})$ for $1\leq j\leq d$.  
We say that $f$ is "good" if the distribution of these colored stars is \   $\begin{picture}(10,12)
\put(5,3){\circle*{2}}
\put(5,8){\circle*{2}}
\put(5,-2){\circle*{2}}
\put(0,3){\circle*{2}}
\put(10,3){\circle*{2}}
\put(5,3){\line(0,1){5}}
\put(5,-2){\line(0,1){5}}
\put(5,3){\line(1,0){5}}
\put(0,3){\line(1,0){4}}
\end{picture}_d$ if $1\leq i\leq n$ is random. The number of good colorings is  $H(\begin{picture}(10,12)
\put(5,3){\circle*{2}}
\put(5,8){\circle*{2}}
\put(5,-2){\circle*{2}}
\put(0,3){\circle*{2}}
\put(10,3){\circle*{2}}
\put(5,3){\line(0,1){5}}
\put(5,-2){\line(0,1){5}}
\put(5,3){\line(1,0){5}}
\put(0,3){\line(1,0){4}}
\end{picture}_d)^{n(1+o(1))}$. 
We obtain a $d$-regular graph $G$ with a desired coloring $g$ by using a perfect matching on the set of half-edges such that the second coordinate of each half-edge is equal to the first coordinate of its pair in the mathching. 
Using Lemma \ref{entlem}, we obtain that the number of such perfect matchings is $$PM(nd)H(\begin{picture}(4,10)
\put(2,0){\line(0,1){6}}
\put(2,0){\circle*{2}}
\put(2,6){\circle*{2}}
\end{picture})^{(dn/2)(1+o(1))}H(\begin{picture}(4,10)
\put(2,0){\line(0,1){6}}
\put(2,0){\circle*{2}}
\put(2,6){\circle*{2}}
\end{picture})^{-dn(1+o(n))}.$$ Thus the number of $d$-regular graphs with a desired coloring is
$$\frac{PM(nd)H(\begin{picture}(10,12)
\put(5,3){\circle*{2}}
\put(5,8){\circle*{2}}
\put(5,-2){\circle*{2}}
\put(0,3){\circle*{2}}
\put(10,3){\circle*{2}}
\put(5,3){\line(0,1){5}}
\put(5,-2){\line(0,1){5}}
\put(5,3){\line(1,0){5}}
\put(0,3){\line(1,0){4}}
\end{picture}_d)^{n(1+o(1))}}{H(\begin{picture}(4,10)
\put(2,0){\line(0,1){6}}
\put(2,0){\circle*{2}}
\put(2,6){\circle*{2}}
\end{picture})^{(dn/2)(1+o(1))}d!^n}.$$ Similarly to the proof of Theorem \ref{edgevertex} we obtain that $H(\begin{picture}(10,12)
\put(5,3){\circle*{2}}
\put(5,8){\circle*{2}}
\put(5,-2){\circle*{2}}
\put(0,3){\circle*{2}}
\put(10,3){\circle*{2}}
\put(5,3){\line(0,1){5}}
\put(5,-2){\line(0,1){5}}
\put(5,3){\line(1,0){5}}
\put(0,3){\line(1,0){4}}
\end{picture}_d)\geq H(\begin{picture}(4,10)
\put(2,0){\line(0,1){6}}
\put(2,0){\circle*{2}}
\put(2,6){\circle*{2}}
\end{picture})^{d/2}$. This completes the proof. 

\end{proof}

\subsection{Entropy inequalities and branching Markov chains}

In Theorem \ref{thm1} we gave a sufficient condition for a branching Markov process to be  factor of i.i.d. process. This can not be necessary, as the example of the Ising model shows. The  
Ising model with parameter $\vartheta$ is the particular case where the Markov chain has only two states and the transition 
matrix $Q=\left(
\begin{tabular}{cc} 
$\frac{1+\vartheta}{2}$ & $\frac{1-\vartheta}{2}$\\ $\frac{1-\vartheta}{2}$ & $\frac{1+\vartheta}{2}$
\end{tabular}
\right)$ is symmetric. 
That is, when we propagate the states from the root along the tree, $\frac{1+\vartheta}{2}$ is the probability that we keep 
the current state. The model is called ferromagnetic if  $\vartheta\geq 0$; i.e. if it is more likely to keep the current state than to change it. The Dobrushin coefficient of the Ising model with parameter $\vartheta\geq 0$ is just $\vartheta$. 
Therefore our theorem implies that when $-1/d<|\vartheta|<1/d$, then the ferromagnetic Ising model is a factor of i.i.d. 
process. But a stronger statement is known: the Ising model is a factor of i.i.d. if $-1/(d-1)\leq\vartheta\leq 1/(d-1)$. To prove this, 
one can use that the clusters in the random cluster representation of the Ising model are almost surely finite in this 
regime. See e.g.
Section 3 of \cite{russ} for the details. See also the paper of H\"aggstr\"om, Jonasson and Lyons \cite{russregi} for a generalization of this result to random-cluster and Potts models.

It is also known that the Ising model with parameter $|\vartheta|>1/\sqrt{d-1}$ can not be factor of i.i.d. (not even a weak limit of factor of i.i.d processes)  see \cite{russ} and \cite{cordec}. 
It is an open question whether the Ising model with $1/(d-1)< |\vartheta|\leq 1/\sqrt{d-1}$ is factor of i.i.d. or not (or whether it is limit of 
factor of i.i.d).

For the ferromagnetic Ising model, the parameter $\vartheta$ is equal to the spectral radius of the transition matrix $Q$, which is, in general, the second largest eigenvalue in absolute value after the eigenvalue $1$. 
More generally, the results of \cite{cordec} imply that a branching Markov process is not the weak limit of factor of i.i.d. 
processes if the spectral radius $\varrho$ of its transition matrix $Q$ is larger than $1/\sqrt{d-1}$. We will 
use Theorem \ref{edgevertex} to show that for general branching Markov processes the correlation bound is far from being optimal. 

\begin{theorem}\label{exnotsuff} For every $d\geq 3$ and $\varepsilon>0$ there exists a transition matrix $Q$ such that 
\begin{itemize}
\item its spectral radius is less than $\varepsilon$;
\item the branching Markov process on the $d$-regular tree $T_d$ according to $Q$ is  not a typical process, and hence it is not the weak limit of factor of i.i.d. processes.
\end{itemize}
\end{theorem}
\begin{proof}
Choose a prime $p$ which is equal to 1 modulo 4 and which satisfies $\frac{2\sqrt{p}}{p+1}<\varepsilon$. 
Let $G$ be a $(p+1)$-regular Ramanujan graph (see the definition below) on $k$ vertices such that 
\[k>(p+1)^{\frac{d}{d-2}}.\] Due to Lubotzky, Phillips and Sarnak \cite{lbs}, this is possible. Let $Q$ be the
transition matrix of the simple random walk on the vertices of $G$. (That is, $Q$ is the adjacency matrix of $G$ normalized 
by $p+1$.) Let $r$ be the spectral radius of $G$. 
By the definition of Ramanujan graphs we 
have that $r\leq \frac{2\sqrt{p}}{p+1}< \varepsilon$.

The branching Markov process on $T_d$ according to $Q$ is an invariant process in $I_d(N)$, where $N$ represents the 
vertices of $G$, that is, it has $k$ elements. Since $G$ is regular, the stationary random walk is uniformly distributed 
on its 
vertices, and therefore the vertex entropy of this branching Markov process is just $\ln k$. 

As for the edge entropy: we can choose the first vertex uniformly at random, and then one of its 
$p+1$ neighbors arbitrarily, but the order does not matter. Therefore the edge entropy is $\ln k+\ln (p+1)$. 

From Proposition \ref{edgevertex} we get that if the branching Markov process according to the transition matrix $Q$ was a typical process, 
then the following would be true: 
\[\frac{d}{2}h(\begin{picture}(4,10)
\put(2,0){\line(0,1){6}}
\put(2,0){\circle*{2}}
\put(2,6){\circle*{2}}
\end{picture})\geq (d-1)h(\begin{picture}(4,10)
\put(2,3){\circle*{2}}
\end{picture});\]
\[\frac{d}{2}[\ln k+\ln(p+1)]\geq (d-1)\ln k;\]
\[d\ln(p+1)\geq (d-2)\ln k;\]
\[(p+1)^{d/(d-2)}\geq k.\]

This contradicts the choice of $k$. Therefore the branching Markov process according to $Q$ is not a typical process. 
\end{proof}

\begin{remark} The example of the Potts model shows that the typicallity of a process or the fact whether it is factor of i.i.d. can not be decided based only on the 
number of states and the spectral radius. 
Let $Q_1$ be the 
transition matrix of the Potts model on $k$ states (see e.g. \cite{sly}): with a 
given probability $p$ it stays at the actual state, otherwise it chooses another state uniformly at random. Its 
spectral radius is equal to $1-\frac{pk}{k-1}$. Moreover, it is also known that the Potts model satisfies the Dobrushin condition if $k>2d$ \cite{sokal}.  
By choosing $p$ such that the spectral radius is so small that the previous theorem can be applied,  we get that the branching Markov chain in the previous theorem is not limit of factor of i.i.d., while  Theorem \ref{thm1} implies that the branching Markov process according to $Q_1$ 
is  a factor of i.i.d. process.

\end{remark}

\begin{remark}
We have seen that the entropy inequality can lead to stronger bound than the correlation decay when the number of states is sufficiently large. However, for the Ising model, when $k=2$, the correlation decay bound is stronger than the bound we get from this entropy inequality. 
\end{remark}

\medskip

\subsection{Entropy inequalities and random $d$-regular graphs}

\label{dominating}

In this section we show how to use entropy inequalities to obtain results about random $d$-regular graphs. Our strategy is that we use Theorem \ref{staredge} to show that certain invariant processes can not be typical. Then, by the correspondence principle, we translate this to statements about random $d$-regular graphs. Throughout this section we assume that $d\geq 3$.

We denote by $C$  the degree $d$ star in $T_d$ with root $o$ and leaves $w_1,w_2,\dots,w_d$.
Let $\mu\in I_d(M)$ be an invariant process. If $F$ is a finite subset of $V(T_d)$, then we denote by $\mu_F$ the marginal distribution of $\mu$ restricted to $F$, and by $\nu_F$ the product measure of the marginals of $\mu_F$.
We denote by $t(F)$ the total correlation of the joint distribution of $\mu_F$; that is,  $t(F)=h(\nu_F)-h(F)$.

\begin{proposition} \label{prop:41}

Let $\mu$ be a typical process and suppose that $h(C)-h(C\setminus \{ w_1\})\leq b$ for some $b\geq 0$. 
Then $t(C\setminus \{ w_1\})\leq b\frac{2d-2}{d-2}$ and \[d_{TV}(\mu_{C\setminus \{ w_1\}}, \nu_{C\setminus \{ w_1\}})\leq \sqrt{b(d-1)/(d-2)}. \]
\end{proposition}

\begin{proof}
By Theorem \ref{staredge} and the condition of the proposition we get 
\begin{equation}\label{eq:p41}0\leq h(C)-\frac d2 h(\{o, w_1\})\leq h(C\setminus \{w_1\})-\frac d2 h(\{o, w_1\})+b.\end{equation}
By using a simple upper bound on the  entropy of $C\setminus \{w_1\}$ we get 
\begin{equation*}0\leq h(o)+(d-1)[h(\{o,w_1\})-h(o)]- \frac d2 h(\{o, w_1\})+b.
\end{equation*}
By rearranging and multiplying by $d/(d-2)$, this implies 
\[-\frac d2h(\{o, w_1\})\leq \frac{db}{d-2}-dh(o).\]
Putting this together with inequality \eqref{eq:p41}, we conclude 
\[0\leq h(C\setminus \{w_1\})-dh(o)+\frac{2d-2}{d-2}b.\] 
Since $h(\nu_F)=dh(o)$ for an invariant process if $F$ consists of $d$ vertices, this concludes the proof of the first inequality.

Observe that $t(C\setminus \{ w_1\})=D\big(\mu_{C\setminus \{ w_1\}}||\nu_{C\setminus \{ w_1\}}\big)$, where $D$ denotes the relative entropy. 
Recall that Pinsker's inequality says that $D(P||Q)\geq 2d_{TV}(P, Q)^2$, where $P$ and $Q$ are two probability distributions on the same set. This implies the statement. \hfill $\square$
\end{proof}

\medskip

As a first application of Proposition \ref{prop:41}, we use it in the case of $b=0$. 

\begin{definition} \label{def:rigid}
Let $S$ be a finite set and $\mu\in I_d(S)$ be an invariant process. Assume that $C$ is a degree $d$ star in $T_d$ with root $o$ and leaves $w_1,w_2,\dots,w_d$. We say that $\mu$ is {\it rigid} if 
\begin{enumerate} 
\item   the values on $C\setminus\{w_1\}$ uniquely determine the value on $w_1$;
\item $\mu$ restricted to $C\setminus\{w_1\}$ is not i.i.d. at the vertices. 
\end{enumerate}
\end{definition}
%

%
%

\begin{proposition}\label{apstaredge} If $\mu\in I_d(S)$ is a rigid process, then it is not typical.
\end{proposition}
\begin{proof} The first assumption in Definition \ref{def:rigid} implies that Proposition \ref{prop:41} holds for $\mu$ with $b=0$, and thus we obtain that $\mu_{C\setminus\{w_1\}}=\nu_{C\setminus\{w_1\}}$, which contradicts the second assumption. \hfill $\square$  
\end{proof}

\medskip

We give an example for families of rigid processes. 

\begin{lemma}\label{rig1} Assume that $S$ is a finite set in $\mathbb{R}$ and that $\mu$ satisfies the eigenvector equation; namely, that a $\mu$-random function $f:T_d\rightarrow S$ satisfies that $\lambda f(o)=f(w_1)+f(w_2)+\dots+f(w_d)$ holds with probability $1$. Then $\mu$ is rigid.
\end{lemma} 
\begin{proof}
Observe that $f(w_1)=\lambda f(o)-(f(w_2)+f(w_3)+\dots+f(w_d))$, which shows that the first condition is satisfied. We want to exclude the possibility that $f(o), f(w_2), f(w_3), \ldots, f(w_n)$ are identically distributed independent random variables. We can assume that all values in $S$ are taken with positive probability. This means that for every pair $(c_1,c_2)\in S\times S$ we have with positive probability that $f(w_2)=f(w_3)=\dots=f(w_d)=c_1$, $f(o)=c_2$, and thus $f(w_1)=\lambda c_2-(d-1)c_1$. It follows that $\lambda S+(1-d)S\subseteq S$ (using Minkowski sum), which is impossible if $S$ is finite.\hfill $\square$
\end{proof}

\medskip

We give further applications of Proposition \ref{prop:41} in extremal combinatorics.

\begin{definition}\label{def:cover}
Let $G=(V, E)$ be a $d$-regular (not necessarily finite) graph. Let $M: S\times S\rightarrow \mathbb N \cup \{0\}$.  We assume that $\sum_{q\in S} M(s,q)=d$ holds for every $s\in S$. Furthermore, we suppose that  the weighted directed graph with adjacency matrix $M$ is connected. Let $f: V\rightarrow S$ be an arbitrary function. We say that $f$ is a covering at $v\in V$ if 
\[\big |\,\{w\ | \ f(w)=q, w\in N(v)\}\,\big |=M(f(v), q),\] 
where $N(v)$ is the set of neighbors of $v$. 
\end{definition}

\begin{lemma}\label{lem:covrig} Assume that $M: S\times S\rightarrow \mathbb N \cup \{0\}$ is as in the previous definition. Fix $\varepsilon\geq 0$ and $d\geq 3$.  Assume furthermore that  $\mu\in I_d(S)$ is an invariant process such that a $\mu$-random function $f: V(T_d^*)\rightarrow S$ is a covering at the root $o$ with probability $1-\varepsilon$, and the distribution of $f(o)$ is supported on at least two elements. Then the following hold. 
\begin{enumerate}[(a)]
\item $h(C)-h(C\setminus \{ w_1\})\leq \varepsilon \log |S|$. 
\item There exists $\delta=\delta(M, \varepsilon)>0$ such that $\mathbb P(f(o)=s)\geq \delta$ holds for all $s\in S$.
\item By using the notation of Proposition \ref{prop:41}, we have \[d_{TV}(\mu_{C\setminus \{ w_1\}}, \nu_{C\setminus \{ w_1\}})\geq \frac12 (\delta^d-\varepsilon).\] 
\item If $\varepsilon=0$, then $\mu$ is rigid.
\end{enumerate} 
\end{lemma}
\begin{proof} We denote by $A$ the event that $f$ is a covering at $o$, and by $B$ its complement. Then $\mathbb P(B)=\varepsilon$.

\noindent $(a)$ For $\varepsilon=0$: observe that $f(w_1)$ is the unique element $q\in S$ with the following property: 
\[|\,\{w\ | \ f(w)=q, w\in \{w_2, w_3, \ldots, w_d\}\}\,\big |=M(f(o), q)-1,\] 
which depends only on the values of $f$ on $C\setminus \{w_1\}$. Therefore the values on $C\setminus \{w_1\}$ uniquely determine the value on $w_1$, and the two entropies are equal.
Otherwise, conditional entropy with respect to an event with positive probability will be defined as the entropy of the conditional distribution. Then we have 
\[
h(C)=h(C|A)\mathbb P(A)+h(C|B)\mathbb P(B)-\mathbb P(A)\log \mathbb P(A)-\mathbb P(B)\log \mathbb P(B);\]
\[h(C\setminus \{w_1\})=h(C\setminus \{w_1\}|A)\mathbb P(A)+h(C\setminus \{w_1\}|B)\mathbb P(B)-\mathbb P(A)\log \mathbb P(A)-\mathbb P(B)\log \mathbb P(B).\]
If $A$ holds, then by the argument above, the value on $w_1$ is uniquely determined by the other ones. Hence 
$h(C\setminus \{w_1\}|A)=h(C|A)$. On the other hand, $h(C|B)\leq h(C\setminus \{w_1\}|B)+\log |S|$ is a trivial upper bound. Therefore we obtain
\[h(C)-h(C\setminus \{ w_1\})=[h(C|B)-h(C\setminus \{ w_1\}|B)]\mathbb P(B)\leq \varepsilon \log |S|.\]

 \noindent $(b)$ We show that $\delta(M, \varepsilon)\geq \frac{a}{d^k}-\frac{\varepsilon}{d-1}$ holds, where $k$ is the diameter of the directed graph with adjacency matrix $M$. If $s\in S$ has probability $a$, then any of its neighbors $t$ has probability at least $(a-\varepsilon)/d$, due to the following. The probability of the event $D$ that $f(o)=s$ and $f$ is a covering at the root is at least $a-\varepsilon$. Given $D$, the joint distribution of the neighbors is permutation invariant. On the event $D$, the values of $f$ evaluated at the neighbors of the root are exactly the neighbors of $s$ with multiplicity in $M$. Hence the probability that the value of $f$ at a fixed neighbor of the root is $t$ is at least $1/d$ conditionally on $D$. Using the invariance of the process, this proves the lower bound for the probability of $t$. 
 
 We can choose an element $s_0\in S$ which has probability at least $1/|S|$. By induction, we have that an element of distance $m$ from $s_0$ in the directed graph $M$ has probability at least 
 \[\frac{1}{|S|d^m}-\varepsilon \bigg [\frac{1}{d}+\frac{1}{d^2}+\ldots+\frac{1}{d^m} \bigg].\]
 Since every other element in $S$ can be reached by a directed path of length at most $k$ in $M$, the proof is complete. 
 
 \noindent $(c)$ Choose $s_1,s_2\in S$ such that $M(s_1, s_2)\leq d/2$.  The covering property at $o$ implies that the probability of the event $\{f(o)=s_1, f(w_2)=s_2, f(w_3)=s_2, \ldots, f(w_d)=s_2\}$ is zero. That is, this event has conditional probability 0 with respect to $A$. It follows that 
 \[\mathbb P(f(o)=s_1, f(w_2)=s_2, f(w_3)=s_2, \ldots, f(w_d)=s_2)\leq \mathbb P(B)=\varepsilon.\]
On the other hand, by part $(b)$ and invariance, the same event has probability at least $\delta^d$ when we consider $\nu$ restricted to $C\setminus \{w_1\}$ (recall that $\nu$ is the product measure of the marginals). This implies the statement.

\noindent $(d)$ The first property follows from the argument in $(a)$. In addition, we have seen in part $(c)$ that the probability of a given configuration is 0. On the other hand, by $(b)$, the probability of each value is positive. 
 This excludes the possibility that $\mu$ restricted to $C\setminus\{w_1\}$ is i.i.d. \hfill $\square$

\end{proof}

\medskip

For the combinatorial applications, we need the following definition.

\begin{definition} 
Let $G=(V, G)$ be a finite $d$-regular graph, and $M: S\times S\rightarrow \mathbb N\cup \{0\}$ as in definition \ref{def:cover}. For an arbitrary function $g: V\rightarrow S$ let $W\subset V$ be the subset of vertices $v$ at which $h$ is not a covering.  We introduce the quantity $e(g):=|W|/|V|$.     Furthermore, we define the covering error ratio of $G$ with respect to $M$ by \[c(G,M)=\min_{g: V\rightarrow S} e(g).\] 
\end{definition}

It will be important that the covering error ratio can be extended to graphings in a natural way such that the extension is continuous in the local-global topology. Let $\mathcal{G}$ be a graphing on the vertex set $\Omega$. Let $g:\Omega\rightarrow S$ be an arbitrary measurable function. Let $W\subseteq\Omega$ be the set of vertices at which $g$ is not a covering of $M$. We denote by $e(g)$ the measure of $W$. We define $c(\mathcal{G},M)$ as the infimum of $e(g)$ where $g$ runs through all measurable maps $g:\Omega\rightarrow S$. We can also obtain $c(\mathcal{G},M)$ as a minimum taken on processes. For $\mu\in I_d(S)$ let $e(\mu)$ denote the probability that a $\mu$ random function $f:T_d^*\rightarrow S$ is not a covering of $M$ at $o$. Using the fact that $e(\mu)$ is continuous in the weak topology and that $\gamma(\mathcal{G},S)$ is compact in the weak topology we obtain that
\begin{equation}\label{cermin}
c(\mathcal{G},M)=\min_{\mu\in\gamma(\mathcal{G},S)}e(\mu).
\end{equation}

Now we are ready to prove the next combinatorial statement. Recall that $\delta(M, 0)>0$, and hence $\varepsilon_0$ defined in the theorem is also positive.

\begin{theorem}\label{thm:combap} Fix $d\geq 3$ and $M$ as in the definition \ref{def:cover}. Let 
\[\varepsilon_0=\inf\bigg\{\varepsilon>0: \frac 12(\delta(M, \varepsilon)^d-\varepsilon)\leq \sqrt {\varepsilon\log |S|\frac{d-1}{d-2}}\bigg\},\] 
where $\delta(M, \varepsilon)$ is defined in Lemma \ref{lem:covrig} $(b)$.
 Then for every $0<\varepsilon<\varepsilon_0$ the probability  $\mathbb P(c(\mathbb G_i, M)<\varepsilon)$ converges to $0$ as $i\rightarrow \infty$, where $\mathbb G_i$ is a random $d$-regular graph on $i$ vertices. 
\end{theorem}

\begin{proof}  Suppose that the invariant process $\mu\in I_d(S)$ satisfies the conditions of Lemma \ref{lem:covrig} for some $\varepsilon>0$, and it is typical. Part $(a)$ implies that  Proposition \ref{prop:41} can be applied with $b=\varepsilon \log |S|$. Putting this  together with part $(c)$ of the lemma, we obtain
\[\frac 12[\delta(M, \varepsilon)^d-\varepsilon]\leq d_{TV}(\mu_{C\setminus \{ w_1\}}, \nu_{C\setminus \{ w_1\}})\leq \sqrt {\varepsilon\log |S|\frac{d-1}{d-2}}.\]

By equation (\ref{cermin}) it follows that $c(\mathcal{G},M)\geq \varepsilon_0$ holds for every typical graphing in $\overline{X_d}$.  Let $0<\varepsilon<\varepsilon_0$ be an arbitrary real number and and let $Q_\varepsilon=\{\mathcal{G}|c(\mathcal{G},M)\leq\varepsilon\}$. By applying Proposition \ref{prop:corres} for $Q_\varepsilon$, the proof is complete. \hfill $\square$
\end{proof}

\medskip

Theorem \ref{thm:combap} provides a family of combinatorial statements depending on the matrix $M$. 
An interesting application of Theorem \ref{thm:combap} is when $M$ is the adjacency matrix of a $d$-regular simple graph $H$. In this case we obtain that random $d$-regular graphs do not cover (not even in an approximative way) the graph $H$. If we apply Proposition \ref{apstaredge} to such a matrix $M$ we get the following. 
Let $\mu\in I_d(V(H))$ be the invariant  process on $T_d$ that is a covering map from $T_d$ to $H$. Then $\mu$ is not typical and thus it is not in the weak closure of factor of i.i.d processes.

 We show two concrete examples, using only $2\times 2$ matrices,  to illustrate how our general statement of Theorem \ref{thm:combap} is related to known results. Note that in these special cases the literature has better bounds then ours; our goal is only demonstrating the connection between different areas.

\begin{equation*}
M_1=\begin{pmatrix} 0 & d \\ 1 & d-1 \end{pmatrix}~~~,~~~M_2=\begin{pmatrix} 0~~ & d \\ d~~ & 0 \end{pmatrix}
\end{equation*}

The dominating ratio of a finite graph $G$ is the following. Let $m$ be the size of the smallest set of vertices $V'$ of $G$ such that each vertex of $G$ is either in $V'$ or connected to a vertex in $V'$. The dominating ratio is defined as $dr(G)=m/|V(G)|$. It is clear that the dominating ratio of a $d$-regular graph is at least $1/(d+1)$. It is easy to see that the dominating ratio of a $d$-regular graph $G$ is equal to $1/(d+1)$ if and only if $c(G,M_1)=0$. For this particular matrix, one can use a better bound than the general one given in Lemma \ref{lem:covrig}. Namely, as a simple calculation shows, $\delta(M, \varepsilon)=1/(d+1)-\varepsilon/(d+1)$ can be chosen.
Theorem \ref{thm:combap} applied to $M_1$ gives to following combinatorial statement. 

\begin{proposition} For every $d\geq 3$ we define 
\[\varepsilon_0=\inf\bigg\{\varepsilon>0: \frac 12\bigg[\bigg(\frac{1-\varepsilon}{d+1}\bigg)^d-\varepsilon\bigg]\leq \sqrt {\varepsilon\log |S|\frac{d-1}{d-2}}\bigg\}.\] 
Then   $P(dr(\mathbb G_i)<1/(d+1)+\varepsilon)$ converges to $0$ as $i\rightarrow \infty$ for all   $0<\varepsilon<\varepsilon_0$. \end{proposition}

This gives the following for small values of $d$. 

\begin{center}
\begin{tabular}{lcccc}
$d$ & 3&4&5&6\\   \hline
$\varepsilon_0$ & $4.38\cdot 10^{-5}$ & $6.15\cdot 10^{-7}$ & $4.47\cdot 10^{-9}$&$2.08\cdot10^{-11}$ 
\end{tabular}
\end{center}

For $d=3$  Molloy and Reed \cite{molloyreed} gave a much better bound $0.2636$ for the dominating ratio; our result gives $0.2500438$. It would be interesting to improve our bounds for larger $d$ as well.

\medskip

The next application  shows  that random $d$-regular graphs are separated from being bipartite, which was first proved by Bollob\'as \cite{bollind}. To put it in another way, it says that the independence ratio (size of the largest independent set divided by the number of vertices) of a random $d$-regular graph is at most $1/2-\varepsilon_0$ with probability tending to $1$ with the number of vertices for some $\varepsilon_0>0$. We can obtain this by applying Theorem \ref{thm:combap} for the matrix $M_2$. In fact, $\delta(M, \varepsilon)\leq 1/2-\varepsilon$, due to the following argument. One of the states has probability at least $1/2$, let us say state $0$. Fix a neighbor of the root. If the root is in state $0$, and the random function is a covering at $0$, then its neighbor is in state 1. This event has probability at least $1/2-\varepsilon$, hence the probability of 1 is at least $1/2-\varepsilon$. 

Therefore 
\[\varepsilon_0=\inf\bigg\{\varepsilon>0: \frac 12[(1/2-\varepsilon)^d-\varepsilon]\leq \sqrt {\varepsilon\log 2\cdot \frac{d-1}{d-2}}\bigg\}.\]

About the best known bounds, see McKay  \cite{mckay} for small $d$. 
For large $d$, the independence ratio of random $d$-regular graphs is concentrated around $2\log d/d$ \cite{bollind, sly}. Our results do not improve their bounds.

\medskip

\begin{remark} From Lemma \ref{rig1} and Proposition \ref{apstaredge} we obtain that any typical processes $\mu$ (and thus any factor of i.i.d process) that satisfy the eigenfunction equation must take infinitely many values. It would be good to see a finer statement about the possible value distributions. Maybe these distributions are always Gaussian. 
\end{remark}

\begin{remark} The proof of Theorem \ref{thm:combap} makes use of the fact that $c(G,M)$ is continuous in the local-global topology. The continuity of various combinatorial parameters in the Benjamini--Schramm topology was studied in e.g. \cite{miklos1, miklos2, gabor}. In those cases it is also possible to prove combinatorial statements through continuity and the analytic properties of the limit objects.
\end{remark}

\subsubsection*{Acknowledgement.} The authors are grateful to Mikl\'os Ab\'ert and to  B\'alint Vir\'ag for helpful discussions and for organizing active seminars in Budapest related to this topic. The research was supported by the MTA R\'enyi Institute Lend\"ulet Limits of Structures Research Group.


\begin{thebibliography}{99}

\bibitem{miklos1} M. Ab\'ert, P. Csikv\'ari and T. Hubai,  Matching measure, Benjamini-Schramm convergence and the monomer-dimer free energy.	Preprint, arXiv:1405.6740 [math-ph].

\bibitem{miklos2} M. Ab\'ert, P. Csikv\'ari, P. Frenkel and G. Kun, Matchings in Benjamini-Schramm convergent graph sequences. To appear in {\it Trans. Amer. Math. Soc.} 	arXiv:1405.3271 [math.CO].

\bibitem{aldous} D. Aldous and R. Lyons, Processes on unimodular random networks. {\it Electron. J. Probab.} {\bf 12} (2007), no. 54, 1454--1508.

\bibitem{alon} N. Alon and J. Spencer, {\it The probabilistic method} (2008), Wiley, New York.

\bibitem{cordec} \'A. Backhausz, B. Szegedy and B. Vir\'ag,  Ramanujan graphings and correlation decay in local algorithms. To appear in Random Structures Algorithms. DOI: 10.1002/rsa.20562

\bibitem{bayati} M. Bayati, D. Gamarnik and P. Tetali, Combinatorial approach to the interpolation method and scaling limits in sparse random graphs.     {\it Ann. Probab.}
    {\bf 41} (2013), No. 6., 4080--4115.
 
\bibitem{bs}  I. Benjamini and O. Schramm, Recurrence of distributional limits of finite
planar graphs. {\it Electron. J. Probab.} {\bf 6} (2001), No. 23, 1--13.

\bibitem{berger} N. Berger, C. Kenyon, E. Mossel and Y. Peres, Glauber dynamics on trees and hyperbolic graphs. {\it Probab. Theory Related Fields} \textbf{131} (2005), No. 3., 311-340.


\bibitem{boll} B. Bollob\'as, {\it Random Graphs} (2001), Cambridge University Press, 2nd edititon.

\bibitem{bollind}  B. Bollob\'as, The independence ratio of regular graphs. {\it Proc. Amer. Math. Soc. } {\bf 83} (1981), No. 2., 
433--436.


\bibitem{BR} B. Bollob\'as and O. Riordan, Sparse graphs: Metrics and random models. {\it Random Structures Algorithms} {\bf 39} (2011), 1--38.

\bibitem{lewis} L. Bowen, The ergodic theory of free group actions: entropy and the f-invariant. {\it Groups Geom. Dyn.} {\bf 4} (2010), no. 3, 419--432.

\bibitem{pathcoupling} R. Bubley and  M. Dyer, Path coupling: A technique for proving rapid mixing in Markov chains. {\it FOCS ’97: Proceedings of the 38th Annual Symposium on Foundations of Computer Science} (1997), 223--231.



\bibitem{csoka} E. Cs\'oka, B. Gerencs\'er, V. Harangi and B. Vir\'ag, Invariant Gaussian processes and independent sets on regular graphs of large girth.  To appear in {\it Random Structures Algorithms.} DOI:
10.1002/rsa.20547

\bibitem{csokalipp} E. Cs\'oka and G. Lippner, Invariant random matchings in Cayley graphs. Preprint.  	arXiv:1211.2374 [math.CO].

\bibitem{dembo} A. Dembo and A. Montanari, Ising models on locally tree-like graphs. {\it Ann. Appl. Probab.} {\bf 20} (2010), no. 2,  367--783.

\bibitem{sly} J. Ding, A. Sly and N. Sun, Maximum independent sets on random regular graphs. Preprint.  	arXiv:1310.4787 [math.PR]


\bibitem{dobrushin} R. L. Dobrushin, Prescribing a system of random variables by conditional distributions. {\it Theory Probab. Appl.} (1970), {\bf 15} 458--486. 

\bibitem{gabor} G. Elek and G. Lippner, Borel oracles. An analytic approach to constant time algorithms. {\it Proc. Amer. Math. Soc.} {\bf 138} (2010), 2939--2947.

\bibitem{wave} Y. Elon, Gaussian waves on the regular tree. Preprint. arXiv:0907.5065[math-ph].

\bibitem{evans} W. Evans, C. Kenyon, Y. Peres and L. J. Schulman, Broadcasting on trees and the Ising model. {\it Ann. Appl. Probab.} \textbf {10} (2000), no. 2, 410--433.

\bibitem{friedman}
J.~Friedman, \textit{A proof of Alon's second eigenvalue conjecture and related
  problems}, Amer. Mathematical Society (2008)

\bibitem{ornstein} N. Friedman and D. Ornstein, On isomorphism of weak Bernoulli transformations. {\it Adv. Math.} \textbf{5} (1971), 365--394.

\bibitem{damien} D. Gaboriau and R. Lyons, A measurable-group-theoretic solution to von Neumann's problem. {\it Invent. Math.} {\bf 177} (2009), 533--540.

\bibitem{gamarnik} D. Gamarnik and M. Sudan, Limits of local algorithms over sparse random graphs.  Proceedings of the 5-th Innovations in Theoretical Computer Science conference, ACM Special Interest Group on Algorithms and Computation Theory, 2014. 

\bibitem{sudan} D. Gamarnik and M. Sudan, Performance of the Survey Propagation-guided decimation algorithm for the random NAE-K-SAT problem. Preprint.  	arXiv:1402.0052 [math.PR].

\bibitem{glasner} E. Glasner, {\it Ergodic theory via joinings.} American Mathematical Society, 2003.

\bibitem{goldberg} D. A. Goldberg, Higher order Markov random fields for independent sets. Preprint. 	arXiv:1301.1762 [math.PR].

\bibitem{russregi} O. H\"aggstr\"om, J. Jonasson and R. Lyons, Coupling and Bernoullicity in random-cluster and Potts models. {\it Bernoulli} {\bf 8} (2002), no. 3, 275--294.

\bibitem{harangi} V. Harangi and B. Vir\'ag,  Independence ratio and random eigenvectors in transitive graphs.	Preprint. arXiv:1308.5173 [math.PR].

\bibitem{HLSz} H. Hatami, L. Lov\'asz and B. Szegedy, Limits of local-global convergent graph sequences. {\it Geom. Funct. Anal.} {\bf 24} (2014), no. 1, 269--296.

\bibitem{hoppen} C. Hoppen and N. Wormald, Local algorithms, regular graphs of large girth, and random regular graphs. Preprint.  	arXiv:1308.0266 [math.CO].

\bibitem{howard} C. D. Howard, Zero-temperature Ising spin dynamics on the homogeneous tree of degree three. {\it J. Appl. Probab.} {\bf 37} (2000), no. 3, 736--747.

\bibitem{JLR} S. Janson, T. Luczak and A. Rucinski, {\it Random graphs} (2000), John Wiley and Sons, Inc.

\bibitem{kungabor} G. Kun, Expanders have a spanning Lipschitz subgraph with large girth. Preprint.  	arXiv:1303.4982 [math.GR].

\bibitem{markovmixing} D. A. Levin, Y. Peres and E. L. Wilmer, {\it Markov chains and 
mixing times}, American Mathematical Society, 2009.

\bibitem{lovasz} L. Lov\'asz, {\it Large Networks and Graph Limits}, American Mathematical Society, 2012.

\bibitem{spacetime} E. Lubetzky and A. Sly, Information percolation for the stochastic Ising model. To appear in {\it J. Amer. Math. Soc.} arXiv:1401.6065 [math.PR].

\bibitem{lbs} A. Lubotzky, R. Phillis and P. Sarnak, Ramanujan graphs. {\it Combinatorica} {\bf 8} (1988), no. 3, 261--277.

\bibitem{russ} R. Lyons, Factors of iid on trees.  To appear in {\it Combin. Probab. Comput.} arXiv:1401.4197 [math.DS].

\bibitem{nazarov} R. Lyons and F. Nazarov, Perfect matchings as IID factors on non-amenable groups. {\it European J. Combin.} {\bf 32} (2011), 1115--1125. 

\bibitem{monoton1} R. Lyons and A. Thom, Invariant coupling of determinantal measures on sofic groups. To appear in {\it Ergodic Theory Dynam. Systems.}   	arXiv:1402.0969 [math.PR]

\bibitem{mckay} B. D. McKay, Independent sets in regular graphs of high girth, Ars Combin. {\bf 23} (1987), A, 179--185.

\bibitem{monotom2} P. Mester, Invariant monotone coupling nedd not exist. {\it Ann. Probab.} {\bf 41} (2013), no. 3a, 1180--1190. 

\bibitem{molloyreed} M. Molloy and  B. Reed, The dominating number of a random cubic graph. {\it Random Structures Algorithms} {\bf 7} (1995), no. 3, 209--221.

\bibitem{exact} E. Mossel and A. Sly, Exact thresholds for Ising--Gibbs samplers on general graphs. {\it Ann. Probab.} \textbf{41} (2013), no. 1, 294--328.
 
\bibitem{mustazee} M. Rahman and B. Vir\'ag, Local algorithms for independent sets are half-optimal. Preprint. arXiv:1402.0485 [math.PR].

\bibitem{sokal} J. Salas and A. D. Sokal, Absence of phase transition for antiferromagnetic Potts models via the Dobrushin uniqueness theorem. {\it J. Stat. Phys.} 
{\bf 86} (1997), no. 3--4,  551--579.

\bibitem{sly} A. Sly, Reconstruction for the Potts model, {\it Ann. Probab.} \textbf{39} (2011), no. 4, 1365--1406.


\bibitem{weitz} D. Weitz, Combinatorial criteria for uniqueness of Gibbs measures, {\it Random Structures Algorithms} {\bf 27} (2005), no. 4, 445–-475.

\end{thebibliography}
\end{document}